\documentclass[12pt]{amsart}

\title[FSI in the LP--formalism]{Fluid-structure interaction in the Lagrange-Poincar\'e formalism: the Navier-Stokes and inviscid regimes}
\author{Henry O. Jacobs} 
\email{hoj201@gmail.com}

\author{Joris Vankerschaver}
\email{Joris.Vankerschaver@gmail.com}

\address{Department of Mathematics, Imperial College London, London
  SW7 2AZ, United Kingdom}

\usepackage{todonotes}
\usepackage{wasysym} %for zodiacs
\usepackage{tikz}
\usetikzlibrary{arrows,matrix,positioning}
\usepackage{amssymb}
\usepackage{geometry} % see geometry.pdf on how to lay out the page. There's lots.
\usepackage{bibentry}
\nobibliography*

\usepackage{color}
\usepackage[pdftex,bookmarks,colorlinks,breaklinks]{hyperref}  % PDF hyperlinks, with coloured links
\definecolor{dullmagenta}{rgb}{0.4,0,0.4}   % #660066
\definecolor{darkblue}{rgb}{0,0,0.4}
\hypersetup{linkcolor=red,citecolor=blue,filecolor=dullmagenta,urlcolor=darkblue} % coloured links

\newtheorem{thm}{Theorem}[section]
\newtheorem{prop}{Proposition}[section]

\newtheorem{definition}{Definition}[section]

\theoremstyle{remark}
\newtheorem{example}[thm]{Example}

\theoremstyle{remark}
\newtheorem{remark}[thm]{Remark}
\newcommand{\qedremark}{\hfill $\diamond$}

\DeclareMathOperator{\SDiff}{SDiff}
\DeclareMathOperator{\SE}{SE}

\DeclareMathOperator{\ad}{ad}
\DeclareMathOperator{\Ad}{Ad}

\DeclareMathOperator{\vol}{vol}

\DeclareMathOperator{\hor}{hor}
\DeclareMathOperator{\Emb}{Emb}
\DeclareMathOperator{\ns}{ns}

\newcommand{\pder}[2]{\ensuremath{ \frac{ \partial #1 }{\partial #2 }
  } }
\newcommand{\body}{\ensuremath{\mathcal{B}}}
\newcommand{\bulk}{\ensuremath{\aquarius}}
\newcommand{\into}{\ensuremath{\hookrightarrow}}
\newcommand{\restr}[2]{ \ensuremath{ \left. #1 \right|_{#2} } }
\newcommand{\lb}{ \ensuremath{\langle}}
\newcommand{\rb}{ \ensuremath{\rangle}}

\begin{document}
\begin{abstract}
In this paper, we derive the equations of motion for an elastic body interacting with a perfect fluid via the framework of Lagrange-Poincar\'e reduction.
We model the combined fluid-structure system as a geodesic curve on the total space of a principal bundle on which a diffeomorphism group acts.
After reduction by the diffeomorphism group we obtain the fluid-structure interactions where the fluid evolves by the inviscid fluid equations.
Along the way, we describe various geometric structures appearing in fluid-structure interactions: principal connections, Lie groupoids, Lie algebroids, etc.
We finish by introducing viscosity in our framework as an external force and adding the no-slip boundary condition.
The result is a description of an elastic body immersed in a Navier-Stokes fluid as an externally forced Lagrange-Poincar\'e equation.
Expressing fluid-structure interactions with Lagrange-Poincar\'e theory provides an alternative to the traditional description of the Navier-Stokes equations on an evolving domain.
\end{abstract}

\maketitle

%\tableofcontents

%%%%%%%%%%%%%%%%%%%%
\section{Introduction}
%%%%%%%%%%%%%%%%%%%%

In this paper, we consider the motion of an elastic structure interacting with an incompressible fluid.
We derive the equations of motion using Lagrangian reduction \cite{MaSc1993a, MaSc1993b, CeMaRa2001} and discuss various geometric structures that arise in fluid-structure interactions.
While we focus primarily on inviscid flows, we also indicate how our framework may be modified to incorporate viscosity.
By expressing fluid-structure problems in the Lagrange-Poincar\'e formalism, differential geometers can communicate about these problems using notions from Riemannian and symplectic geometry.

\subsection{The bundle picture for fluid-structure interactions}

In his foundational work \cite{Arnold1966}, V. I. Arnold showed that the motion of an ideal, incompressible fluid in a fixed container $M$ may be described by (volume-preserving) diffeomorphisms from $M$ to itself.
When a moving structure is present in the fluid, this description is no longer applicable, since the fluid container changes with the motion of the structure.
In our paper, we describe the fluid-structure configuration space instead as a \emph{principal fiber bundle} whose base space is the configuration space for the elastic medium, while the fibers over each point are copies of the diffeomorphism group, which specify the configuration of the fluid.
This point of view is adapted from a similar description in \cite{LeMaMoRa1986, Radford_thesis, Vankerschaver2009}.

In remark \ref{rmk:groupoid} we show that the bundle picture is equivalent to a description involving a certain Lie \emph{groupoid} of diffeomorphisms.
In particular, our principal bundle can be seen as a source fiber of this Lie groupoid, and this provides a nice analogy with the work of Arnold: whereas the motion of an ideal incompressible fluid in a fixed container $M$ is described by curves on the volume preserving diffeomorphism \emph{group}, $\SDiff(M)$, motions of a fluid-structure system (or a fluid with moving boundaries) are described by curves on a volume-preserving diffeomorphism \emph{groupoid}.

\subsection{Derivation of the fluid-structure equations} 
The bulk of our paper is devoted to the derivation of the fluid-structure equations using the framework of \emph{Lagrange-Poincar\'e reduction}.  In section \ref{sec:Lagrangian}, we describe fluid-structure interactions as a Lagrangian mechanical system on a tangent bundle, $TQ$, and we observe that this system has the group of particle relabelings, $G$, as its symmetry group in Proposition \ref{prop:symmetry} (recall that a particle relabeling map is a volume-preserving diffeomorphism of the reference configuration of the fluid).
Upon studying the variational structure in section \ref{sec:variations} we are able to factor out by this symmetry in section \ref{sec:LP}.
As a result, we obtain a Lagrange-Poincar\'e system on $P = TQ/G$ in Theorem \ref{thm:main1}.
In Theorem \ref{thm:main2} we show that these equations are equivalent to the conventional fluid-structure equations, wherein the fluid satisfies the standard inviscid fluid equations.

We briefly repeat this process for viscous fluids by adding an external viscous friction force to the system and incorporating a no-slip boundary condition.
The resulting Lagrange-Poincar\'e equations are equivalent to the conventional fluid-structure equations given by (see for instance \cite[Section 3.3, Section 5.14]{Batchelor}):
\begin{align}
    \label{eq:EofM1}
    \frac{D}{Dt} \left( \pder{L_B}{\dot{b}} \right) -
    \pder{L_B}{b} &= f_\body, \\
    \label{eq:EofM2}
    \pder{u}{t} + u \cdot \nabla u &= \nu \nabla^2 u - \nabla p, \quad \nabla \cdot u = 0.
\end{align}
Here, $\nu$ is the viscosity coefficient, while $u$ and $p$ are the Eulerian velocity and the pressure fields.  The function, $L_B$, is the Lagrangian of the body and $f_\body$ is the force exerted by the fluid on the
body.
In particular, $f_\body$ is implicitly given by the expression for virtual work
\begin{align*}
  f_B \cdot \delta b = - \int_{\partial \body}{ \delta b \cdot \mathbf{T} \cdot n \, d^2 x} \quad , \quad \forall \delta b
\end{align*}
where $\mathbf{T} = - p \mathbf{I} + \nu \left( \nabla u + \nabla u^T \right)$ is the stress tensor of the fluid and $n(x)$ is the unit external normal to the boundary of the body. 

While it is well-known that fluid-structure systems exhibit fluid particle relabeling symmetry, it is to the best of our knowledge the first time that these equations have been derived through Lagrange-Poincar\'e reduction.
Nonetheless, similar efforts have been made, and this ``gauge theoretic'' flavor to fluid mechanics has been adopted in a number of instances.
For example, \cite{ShapereWilczek1989} developed a gauge theory for swimming at low Reynolds numbers.
This gauge theoretic picture of swimming was then applied to potential flow in \cite{Kelly1998}.
Building upon \cite{Kelly1998}, the equations of motion for an articulated body in potential flow were derived as Lagrange-Poincar\'{e} equations (with zero momentum) in \cite{Kanso2005} (see also \cite{Vankerschaver2009}).
Similarly, \cite{Radford_thesis} generalized the work of \cite{Kelly1998} to arbitrary flows and adopted the Lagrange-Poincar\'{e} equations as valid equations of motion under the assumption that they were equivalent to \eqref{eq:EofM1} and \eqref{eq:EofM2}.

\subsection{Conventions}
We will assume the reader is familiar with manifolds and vector bundles, and we refer to \cite{MTA, KoNo1963, KMS99} for a comprehensive overview.
All maps will be assumed smooth throughout this paper. Given any manifold, $M$, its
tangent bundle will be denoted by $\tau_M:TM \to M$ or by $TM$ for short.  Additionally,
given a map $f:M \to N$, the tangent lift will be given by $Tf:TM \to
TN$. We denote the set of $k$-forms on $M$ by $\Omega^k(M)$.  Given a
vector bundle $\pi_E:E \to M$ we denote the set of $E$-valued
$k$-forms by $\Omega^k(M;E) \equiv \Gamma(E) \otimes \Omega^k(M)$.  If
$\alpha \in \Omega^k(M)$, the exterior
derivative of $\alpha$ will be denoted $d\alpha$ and similarly we may
define the exterior derivative of a $E$-valued form by $d( e \otimes
\alpha) = e \otimes d\alpha$ where $e \in \Gamma(E)$.  Lastly, pullback and
pushforward by a map, $\varphi: M \to N$, will be denoted by $\varphi^*$
and $\varphi_*$ respectively.

%%%%%%%%%%%%%%%%%%%%%%%%%%%%%%%%%%%%%%%%
\section{Lagrangian mechanics in the infinite Reynolds regime} \label{sec:Lagrangian}
%%%%%%%%%%%%%%%%%%%%%%%%%%%%%%%%%%%%%%%%
Let $M$ be a Riemannian manifold, with $\| \cdot \|$ the norm associated to the Riemannian metric and $\mu( \cdot )$ the Riemannian volume form.
Viable candidates for $M$ include compact manifolds with and without boundaries, as well as $\mathbb{R}^d$ \cite{Schwarz1995,Troyanov2009}.
Let $\body$ be a compact manifold and let $\Emb(\body ; M)$ denote the set of embeddings of $\body$ into $M$.
The manifold $\Emb(\body ; M)$ will be used to describe the configuration of the body immersed in $M$.
The tangent fiber $T_b \Emb(\body ; M)$ over a $b \in \Emb(\body ; M)$ is given by an embedding $v_b : \body \into TM$ such that $\left. v_b \circ b^{-1} \right|_{b(\body)} : b(\body) \to T_{b(\body)}M$ is a vector-field over $b(\body)$.
The configuration manifold for the body is a submanifold $B \subset \Emb(\body ; M)$, and the Lagrangian of the body is a function $L_B: TB \to \mathbb{R}$.

\begin{figure}[ht]
   \centering
   \includegraphics[width=0.4\textwidth]{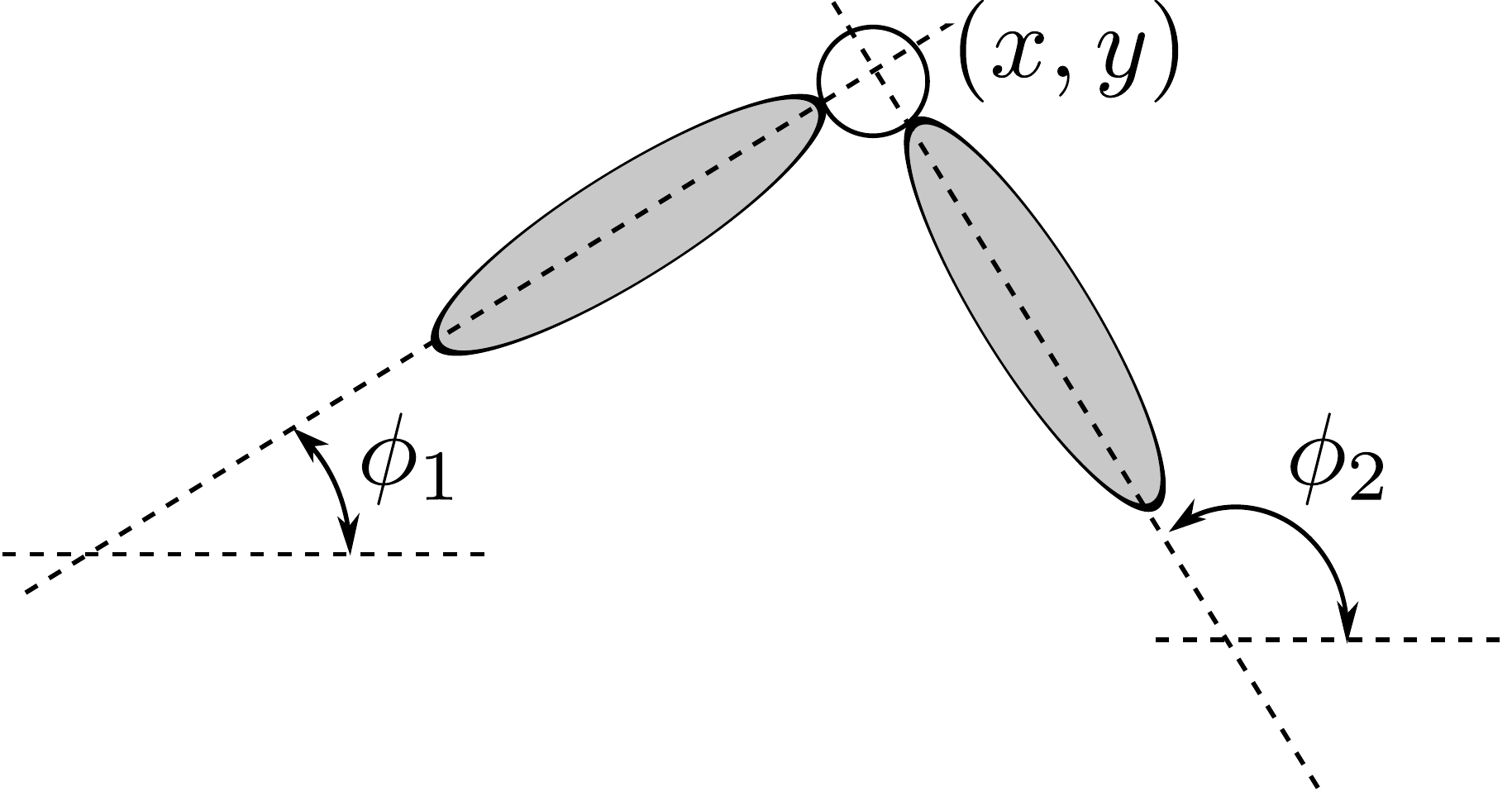} % requires the graphicx package
   \caption{Two rigid bodies connected by a hinge and immersed in $\mathbb{R}^2$.}
   \label{fig:two_link}
\end{figure}

\begin{example}
  Let $M = \mathbb{R}^3$.  The theory of linear elasticity assumes $\body$ to be a Riemannian manifold with a mass density $\rho( \cdot ) \in \bigwedge^3 (\body)$ and metric $\langle \cdot , \cdot \rangle_{\body}: T\body \oplus T\body \to \mathbb{R}$.
  Here the configuration manifold is $B = \Emb(\body ; \mathbb{R}^3)$ and the potential energy
  \[
  	U(b) = \frac{1}{2} \int_{b(\body)}{ \mathrm{trace}\left( [I - C_b]^T \cdot [I - C_b] \right) dx \wedge dy \wedge dz}.
  \]
  where $C_b$ is the push-forward of the metric $\lb \cdot , \cdot \rb_{\body}$ by $b:\body \hookrightarrow \mathbb{R}^d$, a.k.a. the \emph{right Cauchy-Green strain tensor}.
  The $\SE(d)$-invariant kinetic energy, $K:TB \to \mathbb{R}$, is given by
  \[
  	K(b,\dot{b}) = \frac{1}{2} \int_{\body} \big \| \dot{b}(x) \big \|^2 \rho(x).
  \]
  See \cite{MFOE} for more details. \qedremark
  \end{example}
  
  \begin{example} \label{ex:two_link}
 Consider two rigid bodies in $\mathbb{R}^2$ joined by a hinge.
  The configuration manifold, $B$, consists of rigid embeddings of the two links into $\mathbb{R}^2$ such that the embeddings respect the constraint that the links be joined at the hinge (see Figure \ref{fig:two_link}).
  In particular $B$ is isomorphic to a subset of $S^1 \times S^1 \times \mathbb{R}^2$ if we let the tuple $(\phi_1,\phi_2,x,y ) \in S^1 \times S^1 \times \mathbb{R}^2$ denote a configuration where $\phi_1,\phi_2 \in S^1$ are the angles between the links and the $x$-axis, while $(x,y) \in \mathbb{R}^2$ is the location of the hinge.
  
  We may consider a potential energy derived from a linear spring between the hinges given by $U( \phi_1, \phi_2, x,y) = \frac{k}{2} ( \phi_1 - \phi_2 - \bar{\phi} )^{2}$ for some $\bar{\phi} \in S^1$.
  The kinetic energy of the $i^{\rm th}$ body is
  \[
  	K_i = \frac{I_i}{2} \dot{\phi}_i^2 + \frac{M_i}{2} \left( [\dot{x} - \sin(\phi_i) \dot{\phi_i}]^2 + [\dot{y} + \cos(\phi_i) \dot{\phi_i} ]^2 \right)
  \]
  where $M_i$ and $I_i$ are the mass and rotational inertial of the $i^{\rm th}$ body.  The Lagrangian is therefore $L_B = K_1 + K_2 - U$. \qedremark
\end{example}

For a given $b \in B$ the fluid domain is given by the set
\[
	\bulk_b := \mathrm{closure} \{ x \in M : x \notin b(\body)\}.
\]
Fix a reference configuration $b_{\rm ref} \in B$ and define the reference location of the body by $\body_{\rm ref} := b_{\rm ref}(\body) \subset M$ and the reference container for the fluid by $\bulk_{\rm ref} = \bulk_{b_{\rm ref}} \subset M$.
If the body is described by configuration $b \in B$, then the configuration of the fluid is given by a volume preserving diffeomorphism from $\bulk_{\rm ref}$ to $\bulk_{b}$.
We denote the set of volume preserving diffeomorphisms from $\bulk_{\rm ref}$ to $\bulk_{b}$ by $\SDiff( \bulk_{\rm ref}; \bulk_{b})$.
The configuration manifold for the system is
\begin{equation} \label{fsi_config}
	Q := \{ (b ,\varphi) \mid b \in B , \varphi \in \SDiff(\bulk_{\rm ref} ; \bulk_b ) \}.
\end{equation}
Note that for each $(b,\varphi) \in Q$ both $b \circ \left. b_{\rm ref}^{-1} \right|_{\body_{\rm ref}} : \body_{\rm ref} \to b(\body) \subset M$ and $\varphi: \bulk_{\rm ref} \to \bulk_b$ take the boundary $\partial \bulk_{\rm ref} = \partial \body_{\rm ref}$ to the boundary $\partial \bulk_b = \partial \{ b(\body) \}$.
However, it is generally not the case that $b \circ \left. b_{\rm ref}^{-1} \right|_{\body_{\rm ref}}$ maps the boundary to the same points as $\varphi$.
This allows the fluid to slip along the boundary, as is permissible in the infinite Reynolds regime.

Finally, let $G = \SDiff( \bulk_{\rm ref} ; \bulk_{\rm ref})$.  Note that $G$ is a Frechet Lie group with group multiplication given by composition of diffeomorphisms. 
The Lie algebra of $G$ is the set $\mathfrak{X}_{\mathrm{div}}(\bulk_{\mathrm{ref}})$ of divergence-free vector fields on the reference space $\bulk_{\mathrm{ref}}$. The Lie bracket on $\mathfrak{X}_{\mathrm{div}}(\bulk_{\mathrm{ref}})$ will be denoted by $[ \cdot, \cdot ]$ and is the negative of the usual Jacobi-Lie bracket of vector fields: $[u, v] = - [u, v]_{JL}$.
The following proposition can be observed directly.

\begin{prop}
	For each $\psi \in G$ the map $(b , \varphi) \in Q  \mapsto (b , \varphi \circ \psi) \in Q$
	defines a right Lie group action of $G$ on $Q$.
	This action endows $Q$ with the structure of a right principal $G$-bundle.
	Moreover, $B = Q / G$ and the bundle projection is $\pi_B^Q: (b,\varphi) \in Q \mapsto b \in B$.
\end{prop}

The tangent bundle of $Q$ is given by quadruples $(b,\dot{b},\varphi, \dot{\varphi})$ where $(b,\dot{b}) \in TB$, $(b,\varphi) \in Q$ and $\dot{\varphi} : \bulk_{\rm ref} \to T \{ \bulk_b \}$ is such that $\dot{\varphi} \circ \varphi^{-1} : \bulk_b \to T\{ \bulk_b \}$ is a divergence free vector field over $\bulk_b$ (although not generally tangential to the boundary).

\begin{prop} \label{prop:symmetry}
	For each $\psi \in G$ the tangent lift of the action of $G$ on $Q$ is given by
	\[
		(b , \dot{b}, \varphi, \dot{\varphi}) \in TQ  \mapsto (b , \dot{b}, \varphi \circ \psi , \dot{\varphi} \circ \psi ) \in TQ.
	\]
	The quotient space $P := TQ /G$ is given by
	\begin{align}
		P = \left\{ (b,\dot{b},u) \left \arrowvert \begin{array}{l} (b,\dot{b}) \in TB , u \in \mathfrak{X}_{\rm div}( \bulk_b ) \\
			\left. \left( u- \dot{b} \circ b^{-1}  \right) \right|_{\partial \bulk_b} \in \mathfrak{X}( \partial \bulk_b)
			\end{array} \right. \right\}. \label{eq:P}
	\end{align}
	This action endows $TQ$ with the structure of a right principal $G$-bundle with the quotient projection $\pi_{P}^{TQ}: TQ \to P$ described by $\pi^{TQ}_P( b,\dot{b},\varphi, \dot{\varphi}) = (b,\dot{b},u)$ where $u = \dot{\varphi} \circ \varphi^{-1}$.
\end{prop}
The boundary condition in \eqref{eq:P} is called the ``no-penetration condition'', and prevents fluid from entering the domain of the body.

\begin{remark} \label{rmk:groupoid}
One can alternatively view the natural space of fluid structure interaction to be a Lie groupoid.
This perspective can be interpreted as more natural, in the sense that it requires no choice of a reference configuration.
To begin, consider the Frechet Lie groupoid
\[
	\mathcal{G} = \{ (b_1, \varphi, b_0) \mid b_1,b_0 \in B, \varphi \in \SDiff( \bulk_{b_0}; \bulk_{b_1}) \}
\]
equipped with the source map $\alpha( b_1, \varphi, b_0) = b_0$, target map $\beta(b_1,\varphi,b_0) = b_1$, and composition $(b_2, \psi, b_1) \circ (b_1, \varphi, b_0) = (b_2, \psi \circ \varphi, b_0)$.
If $\mathcal{G}$ is transitive then the Lie algebroid of $\mathcal{G}$ is $P$ and $\mathcal{G}$ is the Atiyah groupoid associated to the principal bundle $\pi^Q_B:Q \to B$.
However, it is not generally the case that $\mathcal{G}$ acts transitively.
This could happen for example if $B$ contains embeddings of different volumes.
In any case, $Q$ is the source fiber $\alpha^{-1}( b_{\rm ref})$. \qedremark
\end{remark}

We may now describe the kinetic energy of the fluid by
\[
	\ell_{\mathrm{fluid}}(b,\dot{b},u) = \frac{1}{2} \int_{\bulk_b} \| u(x) \|^2  \, \mu(x),
\]
The total Lagrangian for the fluid structure system is a map $L : TQ \to \mathbb{R}$ given by
\begin{equation} \label{fsi_lagrangian}
	L(b,\dot{b},\varphi, \dot{\varphi}) = L_{\mathcal{B}}(b,\dot{b}) + \ell_{\mathrm{fluid}}(b,\dot{b},u) \quad , \quad u = \dot{\varphi} \circ \varphi^{-1}.
\end{equation}
In other words $L = (L_\mathcal{B} \circ T\pi^Q_B) + (\ell_{\mathrm{fluid}} \circ \pi^{TQ}_P)$.
By construction $L$ is $G$-invariant under the action of $G$ on $TQ$.
The equations of motion on $TQ$ are then given by the Euler-Lagrange equations, or equivalently Hamilton's variational principle.
The conserved momenta of Noether's theorem is the circulation of the fluid, just as it is in the case of fluids without structures immersed in them.
In fact, a proof of this proceeds along the same lines of that of the classical Kelvin circulation found in \cite{ArKh1992}.

%%%%%%%%%%%%%%%%%%%%%%%%%%%%%%%%%%%%%%%%%%%%
\section{Variational structures} \label{sec:variations}
%%%%%%%%%%%%%%%%%%%%%%%%%%%%%%%%%%%%%%%%%%%%
We have found that our system has a symmetry by the Lie group
$G$ of volume-preserving diffeomorphisms, and it is reasonable to assert the existence of a flow on
the quotient space $P$ which is $\pi^{TQ}_{P}$-related to the flow of the Euler-Lagrange equations on $TQ$.
As the admissible variations on $TQ$ yield equations of motion on $TQ$, we will find that admissible variations on $P$ will yield equations of motion on $P$.
This section is devoted to writing down the admissibility condition for variations in $P$.

%%%%%%%%%%%%%%%%%%%%%%%%%%%
\subsection{Horizontal and vertical splittings}
%%%%%%%%%%%%%%%%%%%%%%%%%%%
It is notable that $P$ inherits a vector bundle structure from that of $TQ$.
In particular, $P$ is a vector bundle over $B$ with the projection $\tau : (b,\dot{b},u) \in P \mapsto b \in B$.
This map $\tau$ is nothing more than the push-forward of the tangent bundle projection $\tau_Q:TQ \to Q$ via the commutativity relation $\tau \circ \pi^{TQ}_P = \pi^{Q}_{B} \circ \tau_Q$.
Additionally, there is a natural map $\rho: (b,\dot{b},u) \in P \mapsto (b,\dot{b}) \in TB$ defined by the commutative relation $\rho \circ \pi^{TQ}_{P} = T\pi^{Q}_{B}$.\footnote{Readers who recognized $P$ as an Atiyah algebroid over $B$ should also recognize $\tau$ and $\rho$ as the projection and anchor maps.}

The vector bundle, $P$, can be seen, roughly speaking, as the product of two separate spaces, describing the body and fluid velocities.
In order to express this splitting we may choose to decompose $P$ into complementary vector bundles $P = H(P) \oplus \tilde{\mathfrak{g}}$ where the vertical space is given canonically by $\tilde{\mathfrak{g}} := \mathrm{ker}(\rho)$.
Note that $\tilde{\mathfrak{g}}$ is equivalent to the associated bundle 
\[
	\tilde{\mathfrak{g}} := (Q \times \mathfrak{g}) / G = \{ (b,\xi_b) \mid \xi_b \in \mathfrak{X}_{\rm div}( \bulk_b ) , \left. \xi_b \right|_{\partial \bulk_b} \in \mathfrak{X}( \partial \bulk_b) \}.
\]
In other words, $\tilde{\mathfrak{g}}$ is the vector bundle of Eulerian fluid velocity fields which leave the body fixed.
This vertical space is a Lie-algebra bundle when equipped with the fiber-wise Lie bracket $[(b,\xi_b) , (b,\eta_b) ] := (b, -[\xi_b,\eta_b]_{JL})$.
For each $b \in B$ we may define the adjoint map $\ad_{\xi_b} : \tilde{\mathfrak{g}} \to \tilde{\mathfrak{g}}$ by $\ad_{\xi_b}(b,\eta_b) = (b,[\xi_b,\eta_b])$.
We call the dual map $\ad^*_{\xi_b}: \tilde{\mathfrak{g}}^* \to \tilde{\mathfrak{g}}^\ast$ the \emph{coadjoint map}.

In contrast to $\tilde{\mathfrak{g}}$, there does not exists a natural choice for $H(P)$.
The only requirement on $H(P)$ is that it is complementary to $\tilde{\mathfrak{g}}$.
Therefore, one way to define an $H(P)$ is by choosing a section
$\sigma: TB \to P$ which satisfies the property $\rho
\circ \sigma = id$.  Upon choosing $\sigma$ we set $H(P) = \mathrm{range}(\sigma)$.
Conversely, if we are given a horizontal space $H(P) \subset P$, we may define the section $\sigma = (\restr{\rho}{H(P)})^{-1}$.
  
  For $\sigma:TB \to P$ to be a section of $\rho$ means that for each $(b,\dot{b}) \in TB$ there is a velocity field $u^H$ given by $(b,\dot{b},u^H) = \sigma(b,\dot{b})$.
  In particular, $u^H \in \mathfrak{X}_{\vol}( \bulk_{b})$ is a velocity field which satisfies the boundary condition \eqref{eq:P}.
  Thus we may alternatively view $\sigma$ as a map which assigns $(b,\dot{b}) \mapsto u^H \in \mathfrak{X}_{\vol}(\bulk_b)$. 
This identification of $\sigma$ as a map to Eulerian velocity fields will reduce clutter in future derivations and so this is how we will interpret $\sigma$.
Note that the vector fields produced by $\sigma$ are not tangential to the boundary of the fluid domain.
We may consider the vector bundle over $B$ consisting of pairs $(b,v)$ where $v \in \mathfrak{X}_{\rm div}( \bulk_b) $ is generally \emph{not} tangential to the boundary.
Then $\sigma$ is a vector-bundle valued one form.

\begin{figure}[h] %  figure placement: here, top, bottom, or page
   \centering
   \includegraphics[width=3in]{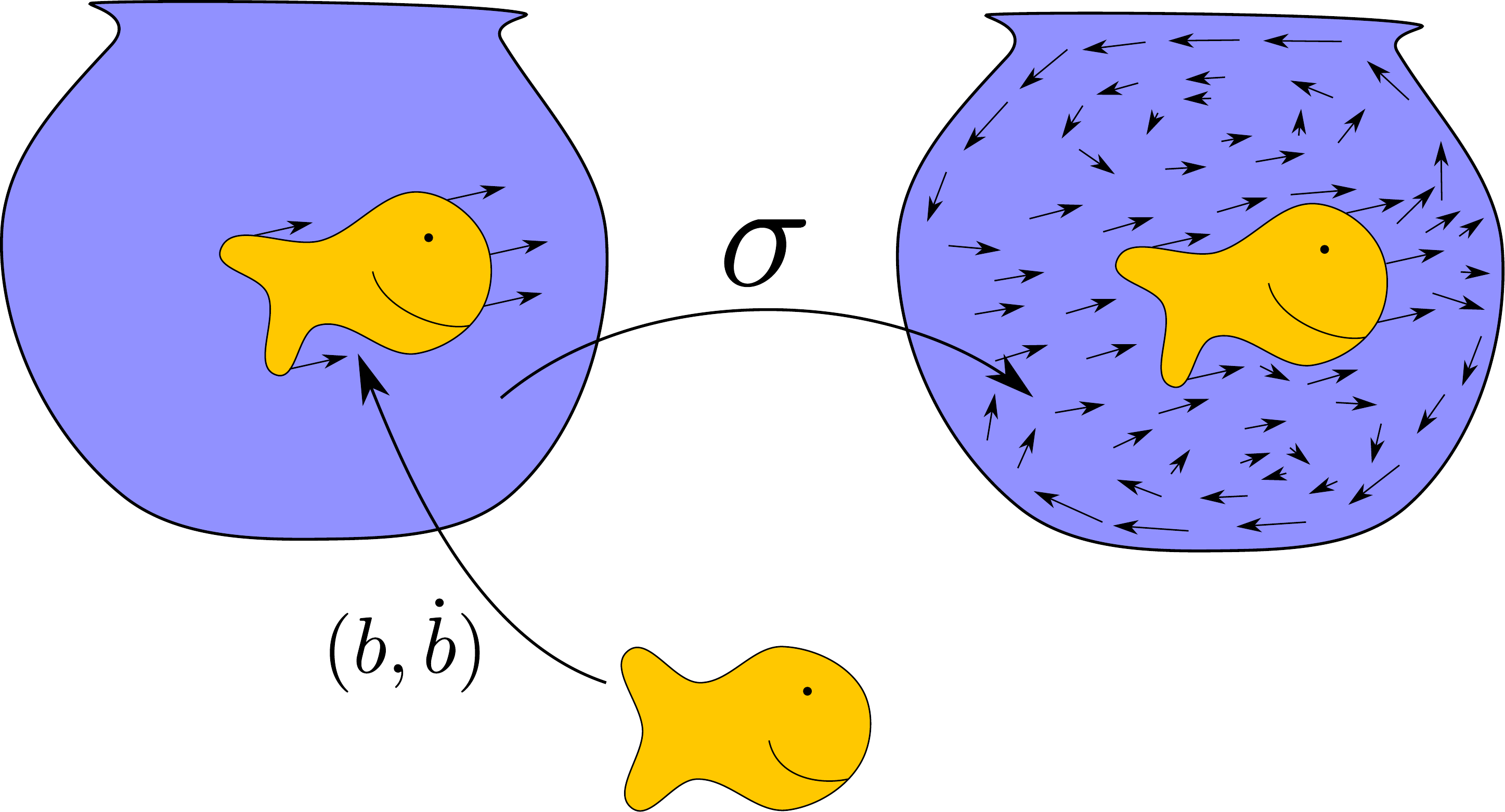} 
   \caption{The map $\sigma$ assigns to each body configuration $b$ and body velocity $\dot{b}$ a velocity field $u^H = \sigma(b, \dot{b})$ for the fluid, which is such that the no-penetration boundary condition holds.}
   \label{fig:section}
\end{figure}

\begin{example}
  One particular choice for $\sigma$ is singled out on physical grounds.  For an inviscid fluid, we begin with the observation that the vorticity is conserved throughout the evolution of the fluid-structure system. As a consequence, the Eulerian velocity field $u$ for the fluid resulting from the body motion given by $\dot{b}$ is a gradient flow, $u = \nabla \Phi$, where the potential $\Phi$ satisfies the following Neumann problem:
\begin{equation} \label{neumann_eq}
	\Delta \Phi = 0, \quad \text{and} \quad 
		\frac{\partial \Phi}{\partial n}(x) = \left<(\dot{b} \circ b^{-1})(x),  n(x)\right> 
			\quad \text{for all $x \in b(\partial \body)$};
\end{equation}
see \cite{La1945, Batchelor}.
The solution of this equation is uniquely defined, up to a constant, and we now define a section of the anchor map denoted $\sigma_{\o}$ by putting 
\begin{equation} \label{neumann_section}
	\sigma_{\o}(\dot{b}) = \nabla \Phi.
\end{equation}

This section $\sigma_{\o}$ can be viewed as the Eulerian version of the ``Neumann connection'' first introduced in  \cite{LeMaMoRa1986} and further analyzed in \cite{Vankerschaver2009}.  We will return to this observation in Appendix~\ref{app:connections}, where we will show that there is in fact a one-to-one correspondence between sections of $\rho$ and principal $G$-connections on $TQ$. 
\qedremark
\end{example}

%%%%%%%%%%%%%%%%%%%%%%%%%%%%%%%%%%%%%%%%%%%%%%%%%%%%%%%%%%%%
\subsection{The covariant derivative on $\tilde{\mathfrak{g}}$}
%%%%%%%%%%%%%%%%%%%%%%%%%%%%%%%%%%%%%%%%%%%%%%%%%%%%%%%%%%%%
  Given a curve $\lambda \mapsto b(\lambda) \in B$ and a section $\sigma:(b,\dot{b}) \in TB \mapsto \sigma(b,\dot{b}) \in \mathfrak{X}_{\rm div}( \bulk_b)$, we obtain a $\lambda$-dependent sequence of vector fields on an evolving domain,
\[  
  u_\lambda = \sigma \left( b(\lambda), \frac{d b}{d \lambda} \right).
\]  
We let $\varphi_{\lambda} \in \SDiff( \bulk_{b(0)};  \bulk_{b(\lambda)})$ be the flow of $u_\lambda$, i.e. $\frac{d\varphi_\lambda}{d \lambda} = u_\lambda \circ \varphi_\lambda$.
We may define the parallel
transport of $(b(0),\xi_{b(0)}) \in \tilde{\mathfrak{g}}$ along $b(\lambda)$ to be the
$\lambda$-dependent curve $(b(\lambda) ,  (\varphi_\lambda)_*\xi_{b(0)})
\in \tilde{\mathfrak{g}}$.  This means of parallel transport allows us to define a
covariant derivative. By taking the derivative of the parallel transport we obtain the element
\begin{equation}
  \restr{\frac{d}{d\epsilon}}{\epsilon = 0}(b_\epsilon , (\varphi_{\epsilon})_* \xi_b) = ((b, \xi_b), (\delta b, [\sigma(b,\delta b) , \xi_b])) \in T(\tilde{\mathfrak{g}}) \label{eq:der_cov_der}
\end{equation}
where $\delta b = \restr{ \frac{db}{d\lambda} }{\lambda = 0}$.\footnote{Recall that the bracket on the right-hand side of \eqref{eq:der_cov_der} is the Lie algebra bracket on $\mathfrak{X}_{\mathrm{div}}(\bulk_{\mathrm{ref}})$, which is the negative of the usual Jacobi-Lie bracket of vector fields.}
The elements of the form \eqref{eq:der_cov_der} define a horizontal space on $T(\tilde{\mathfrak{g}})$ and induce the horizontal and vertical projections
\begin{align}
	\mathrm{hor}((b, \xi_b), (\dot{b}, \dot{\xi}_b)) & = ((b, \xi_b), (\dot{b}, [\sigma(b,\dot{b}), \xi_b ] )) \nonumber \\
	\mathrm{ver}((b, \xi_b), (\dot{b}, \dot{\xi}_b)) & = ((b, \xi_b), (0,\dot{\xi}_b - [\sigma(b,\dot{b}), \xi_b])). \label{induced_ver}
\end{align}

\begin{prop} \label{prop:cov_deriv}
  The vertical projector of \eqref{induced_ver} induces the covariant derivative
\begin{equation} \label{eq_cov_deriv}
\frac{D}{Dt}( b, \xi_b) =  (b, \dot{\xi}_b - [ \sigma(b,\dot{b}), \xi_b] ).
\end{equation}
\end{prop}

\begin{proof}
This follows from Definition \ref{def:connection} (page \pageref{def:cov_der}) and Definition \ref{def:cov_der} (page \pageref{def:cov_der})  for the covariant derivative associated to a connection.
\end{proof}

In Appendix~\ref{app:connections}, we show \eqref{eq_cov_deriv} is also the formula for the covariant derivative induced by a (right) principal connection (see \cite{CeMaRa2001, JaRaDe2012}).
Additionally, the covariant derivative on $\tilde{\mathfrak{g}}$ induces a covariant derivative on the dual bundle, 
$\tilde{\mathfrak{g}}^{\ast}$, by equation \eqref{eq:cov_der2}.  The dual bundle $\tilde{\mathfrak{g}}^*$ can be realized by pairs $(b,\alpha_b)$ where $\alpha_b$ is a covector field on
$\bulk_b$ in the dual space to $\mathfrak{X}_{\vol}(\bulk_b)$, and the induced covariant derivative
is given by
\begin{align}
\frac{D (b,\alpha_b) }{Dt} = (b , \dot{\alpha}_b + \ad_{\sigma(b,\dot{b})}
^* \alpha_b ) \label{eq:dual_cov_der}.
\end{align}
However, the (smooth) dual to $\tilde{\mathfrak{g}}$ is \emph{not} the space of covector fields on $\bulk_b$.
Strictly speaking, the dual space is the set of equivalences classes consisting of one-forms modulo exact forms on $\bulk_b$ \cite{ArKh1992}, but we will not use this identification.

%%%%%%%%%%%%%%%%%%%%%%%%%%%%%%%%%%%%%%%%%%%%%%%%%%%%%%%%%%%%
\subsection{The isomorphism between $P$ and $TB \oplus \tilde{\mathfrak{g}}$.}
%%%%%%%%%%%%%%%%%%%%%%%%%%%%%%%%%%%%%%%%%%%%%%%%%%%%%%%%%%%%
We can construct an isomorphism between the quotient bundle $P$ and $TB\oplus \tilde{\mathfrak{g}}$.
From a fluid-dynamical point of view, this isomorphism decompose the fluid velocity field into a part consistent with a moving body and a part consistent with a stationary body.

\begin{prop}\label{prop:bundle_isomorphism}
  Let $\sigma : (b,\dot{b}) \in TB \mapsto u^H \in \mathfrak{X}_{\rm div}( \bulk_b)$ be a section of $\rho$.  Then the map $\Psi_{\sigma} : P \to TB \oplus \tilde{\mathfrak{g}}$ given
  by $\Psi_{\sigma}(b,\dot{b},u) = (b , \dot{b},u - \sigma(b,\dot{b}) )$ is a vector-bundle isomorphism.
\end{prop}

\begin{proof}
 To check that $\Psi_{\sigma}$ maps to the correct space it suffices to check that $u - \sigma(b,\dot{b})$ is a vector field which is tangent to the boundary of $\bulk_b$.  However this is true by construction because both $\sigma(b,\dot{b})$ and $u$ must satisfy the no-penetration condition given in  \eqref{eq:P}, so that the difference $\sigma(b,\dot{b}) - u$ is tangent to the boundary.
  Additionally, $\Psi_{\sigma}$ is
  invertible with the inverse $\Psi_{\sigma}^{-1}(b,\dot{b},\xi_b) = (b,\dot{b}, \xi_b + \sigma(b,\dot{b}))$.
\end{proof}

\begin{example}
Let $\sigma$ be defined by the Neumann map \eqref{neumann_section}.
In this case, $\sigma_{\o}(b,\dot{b}) = \nabla \Phi$.  We then have that 
\[
	\dot{\varphi} \circ \varphi^{-1} = \underbrace{(\dot{\varphi} \circ \varphi^{-1} - \nabla \Phi)}_{\text{consistent with stationary body}} + \underbrace{(\nabla \Phi)}_{\text{consistent with moving body}},
\]
where $\Phi$ is given by \eqref{neumann_eq}.  The first factor, $\dot{\varphi} \circ \varphi^{-1} - \nabla \Phi$, is a divergence-free vector field tangent to the boundary, while the second factor is a gradient vector field, so that this decomposition is nothing but the Helmholtz-Hodge decomposition of vector fields. If we let $\mathbb{P}$ denote the projection onto divergence-free vector fields in the Helmholtz-Hodge decomposition, the isomorphism in Proposition~\ref{prop:bundle_isomorphism} becomes
\[
	\Psi_{\sigma_{\o}}( [ ( b , \dot{b} , \varphi, \dot{\varphi} ) ] ) = (b , \dot{b}
, \mathbb{P}(\dot{\varphi} \circ \varphi^{-1})).
\]
In general, when $\sigma$ does not produce gradient vector fields, the isomorphism $\Psi_\sigma$ can be viewed as giving rise to a \emph{generalized Helmholtz-Hodge decomposition}. \qedremark
\end{example}

%%%%%%%%%%%%%%%%%%%%%%%%%%%%%%%%%
\subsection{Covariant variations}
%%%%%%%%%%%%%%%%%%%%%%%%%%%%%%%%%
  Let $(b,\varphi)_t$ be a curve in $Q$.  Then a \emph{(finite) deformation} is
  a $\lambda$-parametrized family of curves $(b,\varphi)_{\lambda , t}$ such that
  $(b,\varphi)_{0,t} = (b,\varphi)_t$.
 We desire to measure how much the infinitesimal variation $\delta ( b ,
\varphi)_t = \restr{ \pder{}{\lambda} }{ \lambda = 0} ( b,\varphi)_{\lambda,t}$ changes the quantity $(b,\xi_b)_t :=
(b_t , \dot{\varphi}_t \circ \varphi_t^{-1} - \sigma(b,\dot{b})) \in \tilde{\mathfrak{g}}$.  To do this we invoke the covariant
derivative of Proposition \ref{prop:cov_deriv}.

Additionally, we may take the Levi-Civita connection of an arbitrary Riemannian metric on $B$.
The Whitney sum of these covariant derivatives is a covariant derivative on $TB \oplus \tilde{\mathfrak{g}}$ which we also denote $\frac{D}{Dt}$.

Using this covariant derivative, we may define the \emph{covariant variation}
of a curve $(b,\varphi)_t \in \tilde{\mathfrak{g}}$ with respect to a deformation
$(b,\xi_b)_{\lambda,t}$ by 
\begin{equation} \label{covariant_derivative}
\delta^{\sigma}(b,\xi_b) := \restr{\frac{D(b, \xi_b) }{ D \lambda }}{\lambda=0}.
\end{equation} 

We now compute the variation of a reduced curve $(b, \xi_b)_t$ induced by a variation of the curve $(b, \varphi)_t$ in $Q$. We split this computation into two parts: in proposition~\ref{prop:vert_var} we compute the effect of vertical variations, which leave $b$ fixed and act on $\varphi$ by particle relabelings. In proposition~\ref{prop:hor_var} we then consider the effect of horizontal variations, which change $b$ (and change $\varphi$ accordingly).

\begin{prop}\label{prop:vert_var}
  Let $\sigma: (b,\dot{b}) \in TB \mapsto \sigma(b,\dot{b}) \in \mathfrak{X}_{\rm div}(\bulk_b)$ be a section of $\rho$ and let $(b,\varphi)_t$
  be a curve in $Q$.  Define the time-dependent vector field
\[
  (\xi_b)_t = \dot{\varphi}_t \circ \varphi_t^{-1} - \sigma((b,\dot{b})(t)),
\]
so that $(b,\xi_b)_t \in \tilde{\mathfrak{g}}$ for all $t$.  Given a vertical
deformation of the curve given by $(b_t, \varphi_t \circ \psi_{t,\lambda})$ for a time-dependent
 deformation of the identity, $\psi_{t,\lambda} \in G$, the covariant variation of $(b,\xi_b)_t$ is given by
\[
\delta^{\sigma} (b,\xi_b) = \frac{D (b,\eta_b) }{Dt} - [(b,\xi_b),(b
, \eta_b) ]
\]
where $\eta_b = (\varphi_t)_*\eta$ and $\eta := \restr{\pder{
    \psi_{t,\lambda} }{\lambda } }{\lambda=0} \in \mathfrak{g}$.
\end{prop}

\begin{proof}
  To reduce clutter we shall suppress the $t$ and $\lambda$ dependence
  of $\varphi_{\lambda,t}$ and abbreviate it as `$\varphi$'.
  The partial derivative $\delta \xi_b :=\restr{
    \pder{} {\lambda}}{\lambda=0}( \xi_b)$ may be split into three parts.
  \begin{align*}
    \delta \xi_b = \underbrace{\pder{ \dot{\varphi} }{\lambda} \circ \varphi^{-1}}_{T_1} +
    \underbrace{ T\dot{\varphi} \circ \pder{ \varphi^{-1}}{\lambda}
    }_{T_2} - \underbrace{ \pder{
      \sigma(b,\dot{b}) }{\lambda} }_{T_3}.
  \end{align*}
  As the variation is vertical, $\sigma(b,\dot{b})$ does not depend on $\lambda$ and therefore $T_3 = 0$.  We can rewrite $T_1$ as 
  \begin{align*}
    T_1=\pder{\dot{\varphi}}{\lambda} \circ \varphi^{-1} &= \frac{ \partial^2 \varphi }{ \partial
      \lambda \partial t} \circ \varphi^{-1} = \pder{}{t}( \delta \varphi) \circ \varphi^{-1} = \pder{}{t}(T\varphi \circ \eta) \circ \varphi^{-1} \\
      &= T \dot{\varphi} \circ \eta \circ \varphi^{-1} + T \varphi \circ \dot{\eta} \circ \varphi^{-1} = T \dot{\varphi} \circ \eta \circ \varphi^{-1} + \varphi_* \dot{\eta},
  \end{align*}
where we have used the fact that 
\[
	\delta \varphi = \frac{\partial}{\partial \lambda}( \varphi \circ \psi_\lambda) \Big|_{\lambda = 0}
		= T \varphi \circ \eta.
\]  
  
On the other hand, $T_2$ may be written as
\begin{align*}
 T_2= T \dot{\varphi} \circ \pder{ \varphi^{-1} }{ \lambda } &= -T(\dot{\varphi}
  \circ \varphi^{-1} )\circ \delta \varphi \circ \varphi^{-1} = - T(\dot{\varphi} \circ \varphi^{-1}) \circ T\varphi \circ \eta \circ
  \varphi \\
  &= - T\dot{\varphi} \circ \eta \circ \varphi^{-1}
\end{align*}
Thus we find $\delta \xi_b = \varphi_* \dot{\eta}$.  Since $\delta b =
0$, and using the expression \eqref{covariant_derivative} for the covariant derivative, we see that
$\delta^{\sigma} (b, \xi_b) = (b,\delta \xi_b - [\sigma( \delta b) , \xi_b ]) = (b, \varphi_*\dot{\eta}  )$.
Additionally note that
\[
\frac{D(b,\eta_b)}{Dt} = \left(b, \frac{d \eta_b }{dt} - [
\sigma(b,\dot{b}) , \eta_b]\right).
\]
We calculate 
\begin{align*}
\frac{d \eta_b }{dt}=\frac{d}{dt} \varphi_* \eta &= [ \dot{\varphi} \circ \varphi^{-1} ,
\varphi_* \eta] + \varphi_* \dot{\eta} = [ \xi_b + \sigma(b,\dot{b}) , \eta_b ] + \varphi_* \dot{\eta}.
\end{align*}
Therefore $\varphi_* \dot{\eta} = \frac{d
  \eta_b}{dt} - [ \xi_b + \sigma(b,\dot{b}), \eta_b
]$ so that 
\begin{align*}
\delta^{\sigma}(b,\xi_b) &= \left(b, \frac{d
  \eta_b}{dt} - [\xi_b + \sigma(b,\dot{b}) ,
\eta_b] \right) = \frac{D (b,\eta_b)}{Dt} - [(b,\xi_b), (b,\eta_b) ]. \qedhere
\end{align*} 
\end{proof}

We understand how variations of curves in $Q$ induced by $G$ are expressed as variations in $\tilde{\mathfrak{g}}$, but how about other variations?  We
must also consider how $(b,\xi_b) = (b, \dot{\varphi} \circ \varphi^{-1} - \sigma(b,\dot{b}))$ varies in response to variations of
$b$.  Given a curve $(b,\varphi)_t \in Q$ we may take a deformation of
the curve $b_t \in B$ given by $b_{t,\lambda}$ and define the
\emph{horizontal deformation} of $(b,\varphi)_t$ as follows. For fixed $t$, we consider the vector field $u_{t, \lambda}$ defined by $ u_{t, \lambda} = \sigma \left( \frac{\partial b_{t, \lambda}}{\partial \lambda} \right)$,
 and we let $\varphi_{t, \lambda}$ be the flow of $u_{t, \lambda}$, so that $\dot{\varphi}_{t, \lambda} = u_{t, \lambda} \circ \varphi_{t, \lambda}$. In this way, we obtain a family of curves $(b_{t, \lambda}, \varphi_{t, \lambda})$, which we will call a horizontal deformation of $(b, \varphi)_t$.

\begin{prop}\label{prop:hor_var}
  Given a horizontal deformation of $(b,\varphi)$ given by $(b_{t,\lambda}, \varphi_{t,\lambda})$
 the covariant variation of $(b,\xi_b) := (b,\dot{\varphi}\circ \varphi^{-1} - \sigma(b,\dot{b}))$ is
\[
\delta^{\sigma} (b,\xi_b) = C(\dot{b}, \delta b),
\]
where $C$ is a $\tilde{\mathfrak{g}}$-valued $2$-form given by 
\[
  C \left( \dot{b},\delta b \right) = d \sigma \left ( \dot{b}, \delta b \right) - \left [
\sigma ( \dot{b} ) , \sigma \left( \delta b \right) \right].
\]
where we are viewing $\sigma$ as a vector-bundle valued form.
\end{prop}

\begin{remark}
Up to sign, the $\tilde{\mathfrak{g}}$-valued 2-form $C$ is the \emph{curvature} of a principal connection which is identifiable with $\sigma$. We will return to this point in Section~\ref{sec:curvature}, but for now we just refer to $C$ as the curvature of $\sigma$.  In particular, Proposition \ref{prop:hor_var} is an instance of Lemma 3.1.2 of \cite{CeMaRa2001} when adapted to our specific infinite dimensional setting. \qedremark
\end{remark}

\begin{proof}
  Upon invoking Proposition \ref{prop:df=dmf+def} and the fact that
  the fiber derivative, $\pder{\sigma}{\dot{b}}$, is identical to
  $\sigma$ (because $\sigma$ is linear in the fibers) we find that
\[
\delta \xi_b = \restr{ \pder{ \dot{\varphi}}{\lambda} }{ \lambda=0 }
\circ \varphi^{-1} - T\dot{\varphi} \circ \varphi^{-1} \circ \delta \varphi
\circ \varphi^{-1} - \left \lb \pder{ \sigma }{b}(\dot{b}) ,
  \delta b \right \rb  -
 \sigma \left( \frac{D}{Dt}( \delta b ) \right) .
\]
Using the equivalence of mixed partials we find that
\begin{align*}
\restr{ \pder{ \dot{\varphi} }{\lambda } }{ \lambda = 0} \circ \varphi^{-1} &=
\frac{d}{dt}( \delta \varphi ) \circ \varphi^{-1} = \frac{d}{dt} \left( \sigma(
  \delta b) \circ \varphi \right) \circ \varphi^{-1} \\
&= \left \lb \pder{\sigma}{b}(\delta b), \dot{b} \right \rb +
 \sigma \left( \frac{D}{Dt}( \delta b ) \right) +
T(\sigma(\delta b)) \circ  (\dot{\varphi}
\circ \varphi^{-1}).
\end{align*}
Upon substitution into the last line of the previous calculation we
find
\begin{align*}
\delta \xi_b &= \left \lb \pder{\sigma}{b}(\delta b) , \dot{b}
  \right \rb -
  \left \lb \pder{\sigma}{b}(\dot{b}) ,  \delta b \right \rb
+ T(\sigma(\delta b)) \circ ( \dot{\varphi} \circ \varphi^{-1} )
- T(\dot{\varphi} \circ \varphi^{-1} ) \circ \sigma(\delta b) \\
\intertext{and by Proposition \ref{prop:da(v,w)=da(w)v-da(v)w} }
&= d\sigma(\dot{b},\delta b) + [ \sigma(\delta b) ,
\dot{\varphi} \circ \varphi^{-1} ] = d\sigma(\dot{b} ,\delta b) + [ \sigma(\delta b) , \xi_b + \sigma(\dot{b}) ].
\end{align*}
Therefore the covariant variation is
\[
\delta^{\sigma}(b,\xi_b) = (b,\delta \xi_b -
\left [\sigma(b,\delta b),\xi_b \right]) = (b, d\sigma(\dot{b} , \delta b ) - \left[
\sigma( \dot {b} ) ,
\sigma( \delta b ) \right]) :=C(\dot{b},\delta b). \qedhere
\]
\end{proof}

 In summary, a variation $\delta q(t)$ of a curve $q(t) \in Q$ will lead to a variation of the curve $\dot{q}(t) \in TQ$ which can be passed to the quotient $P$.
 Moreover if we choose a section, $\sigma$, we get  an isomorphism from $P$ to the space $TB \oplus \tilde{\mathfrak{g}}$ and the relevant variations take the form
\[
	\delta^{\sigma}( b , \dot{b}, \xi_b) = \left( b, \frac{D \dot{b} }{D\lambda} , \frac{D \eta_b}{Dt} - [ \xi_b, \eta_b] + C(\dot{b},\delta b) \right).
\]
  Now that we fully understand the journey from variations of curves in $Q$ to variations of curves in $TB \oplus \tilde{\mathfrak{g}}$ we may use Hamilton's variational principle to derive the equations of motion on $TB \oplus \tilde{\mathfrak{g}}$.
  These equations are known as the \emph{Lagrange-Poincar\'{e} equations}.

%%%%%%%%%%%%%%%%%%%%
\section{The Lagrange-Poincar\'e equations} \label{sec:LP}
%%%%%%%%%%%%%%%%%%%%
  In this section we will derive the equations of motion on $P \cong TB \oplus \tilde{\mathfrak{g}}$ using a reduced version of Hamilton's principle known as the Lagrange-Poincar\'{e} variational principle.
  We will first derive the equations of motion in geometric form (Theorem~\ref{thm:main1}) and we will then show that these equations are equivalent to the equations \eqref{eq:EofM1} and \eqref{eq:EofM2} in coordinates when $M = \mathbb{R}^d$ (Theorem~\ref{thm:main2}).

First recall that the total fluid-structure Lagrangian, $L:TQ \to \mathbb{R}$, is  $G$-invariant and the reduced Lagrangian $\ell : P \to \mathbb{R}$ is $\ell(b,\dot{b},u) = L_B(b,\dot{b}) + \ell_{\mathrm{fluid}}( b , \dot{b} , u)$ where
\[
	\ell_{\mathrm{fluid}}(b,\dot{b},u) = \frac{1}{2} \int_{\bulk_b} \| u(x) \|^2 \, \mu(x).
\]
In the previous section we found that $P$ is isomorphic to $TB \oplus \tilde{\mathfrak{g}}$ upon choosing a section of $\rho$ given by a map $\sigma:(b,\dot{b}) \in TB \to \sigma(b,\dot{b}) \in \mathfrak{X}_{\rm div}(\bulk_b)$.  The reduced Lagrangian on $T B \oplus \tilde{\mathfrak{g}}$ (which we will also denote by $\ell$) is given by
\begin{align}
\ell( b, \dot{b} , \xi_b ) = L_{B}( b , \dot{b}) + \frac{1}{2} \int_{
  \bulk_{b}}{ \| \xi_b(x) + \sigma(b,\dot{b})(x) \|^2 \, \mu(x) } \label{eq:reduced_lag}.
\end{align}

With this reduced Lagrangian we may transfer Hamilton's principle from $TQ$ to $TB \oplus \tilde{\mathfrak{g}}$ and derive symmetry-reduced equations of motion as in \cite{CeMaRa2001}.

\begin{thm}[Lagrange-Poincar\'{e} Theorem] \label{thm:main1}
  Let $\sigma:(b,\dot{b}) \in TB \to \sigma(b,\dot{b}) \in \mathfrak{X}_{\rm div}(\bulk_b)$ be a section of $\rho:P \to TB$ and let $q(t) =
  (b,\varphi)_t$ be a curve in $Q$.
  Let $(b,\dot{b},\xi_b)(t) = \Psi_{\sigma}( \pi^{TQ}_P( q,\dot{q}) )$ be the induced curve in $TB \oplus \tilde{\mathfrak{g}}$.
  Then the following are equivalent
\begin{enumerate}
\item The curve $q( \cdot )$ extremizes the action functional
\[
\mathcal{A} = \int_{0}^{t}{ L(q,\dot{q}) dt}
\]
with respect to arbitrary variations of $q( \cdot )$ with fixed end points.
\item The curve $q( \cdot )$ satisfies the Euler-Lagrange
  equations
\[
\frac{D}{Dt} \left( \pder{L}{\dot{q}} \right) - \pder{L}{q} = 0.
\]
with respect to an arbitrary torsion-free connection on $TQ$.
\item The curve $(b,\dot{b},\xi_b )(\cdot)$ extremizes the reduced action
\[
\mathcal{A}_{red} = \int_{0}^{t}{ \ell( b ,\dot{b} , \xi_b) dt }
\]
with respect to variations, $\delta b$, of the curve $b( \cdot )$  with fixed end points and
covariant variations of $(b,\xi_b)(\cdot)$ of the form
\begin{align}
	\delta^{\sigma} (b,\xi_b) = \frac{D (b,\eta_b)
}{Dt} - (b, [\xi_b , \eta_b]) + C(\dot{b},\delta b) \label{eq:vert_variations}
\end{align}
 for an arbitrary curve $(b,\eta_b)( \cdot)$ in $\tilde{\mathfrak{g}}$ over the curve $b(\cdot)$.
\item The curve $(b,\dot{b},\xi_b)(\cdot)$ satisfies the
  Lagrange-Poincar\'{e} equations:
  \begin{align}
    &\frac{D}{Dt} \left( \pder{\ell}{ \dot{b} } \right) - \pder{\ell}{b} =
    i_{\dot{b}} C_{\partial \ell / \partial \xi_b} \label{eq:LP_hor}
    \\
    &\frac{D}{Dt} \left( \pder{\ell}{\xi_b} \right) =
    -\ad_{(b,\xi_b) }^* \left(
  \pder{\ell}{\xi_b} \right). \label{eq:LP_ver}
  \end{align}
where $C_{\partial \ell / \partial \xi_b}$ is a two-form on $B$ given by $C_{\partial \ell / \partial \xi_b}( \dot{b},
\delta b) = \left \lb \pder{\ell}{\xi_b} , C(\dot{b},\delta b) \right \rb$.
\end{enumerate}
\end{thm}

\begin{proof}
  The equivalence of (1) and (2) is standard.
  The equivalence of (1) and (3) follows from taking a deformation,
  $(b,\varphi)_{t,\lambda}$, of $(b,\varphi)( t)$ with fixed
  end-points.
  Note that $\xi_b = \dot{\varphi} \circ \varphi^{-1} - \sigma(b,\dot{b})$.
  By Proposition \ref{prop:vert_var} and
  \ref{prop:hor_var} we see that $\delta^\sigma(b,\xi_b) =
  \frac{D}{Dt}(b,\eta_b) + (b,[\xi_b,\eta_b] ) + C(\dot{b},\delta b)$ and $\delta b$ is a variation of $b(t)$ with
  fixed endpoints.  Therefore, if $q(t)$ extremizes $\mathcal{A}$ with
  respect to arbitrary variations with fixed end points, then clearly
  $(b,\dot{b},\xi_b)(\cdot)$ extremizes $\mathcal{A}_{red}$ with respect
  to the variations of the desired form.  Conversely, variations
  $\delta b$ and $\delta^{\sigma} (b,\xi_b) = \frac{D
    }{Dt}(b,\eta_b) + (b, [\eta_b,\xi_b] )+ C(\dot{b},\delta
  b)$ produce arbitrary variations of $(b,\varphi)_t$ by the formula
  $\delta \varphi = \left( \sigma(\delta b) + \eta_b
  \right) \circ \varphi$.  Thus (1) and (3) are equivalent.

  Now we will prove the equivalence of (3) and (4). By assuming
  (3) and taking variations we find
 \begin{align*}
    0 &= \delta \int_{0}^{T}{ \ell(b,\dot{b}, \xi_b) dt } \\
  &= \int_{0}^{T}{ \left \lb \pder{\ell}{ \xi_b } , \delta^{\sigma}
      (b,\xi_b) \right \rb +
    \left \lb \pder{\ell}{b} , \delta b \right \rb + \left \lb \pder{\ell}{\dot{b}} ,
    \frac{D}{D\lambda}( \dot{b}) \right \rb
    dt } \\
  &= \int_{0}^{T}{ \left \lb \pder{\ell}{\xi_b} , \frac{D}{Dt}(b,\eta_b) - [(b,\xi_b),(b,\eta_b)] +
    C(\dot{b},\delta b) \right \rb + \left \lb \pder{ \ell}{b} , \delta b
  \right \rb + \left \lb \pder{ \ell }{\dot{b}} , \frac{D}{Dt}(\delta b)
  \right \rb dt} \\
   &= \int_{0}^{T}{ \left \lb \pder{ \ell }{ \xi_b } ,
       \frac{D}{Dt}(b,\eta_b) - \ad_{(b,\xi_b)}(b,\eta_b) \right \rb
       + \left \lb \pder{ \ell }{b} - \frac{D}{Dt} \left(
       \pder{ \ell }{\dot{b}} \right) + i_{\dot{b}}C_{\partial \ell / \partial \xi_b} , \delta b \right \rb dt } \\
   &= \int_{0}^{T}{ \left \lb -\frac{D}{Dt} \left( \pder{ \ell }{\xi_b} \right) -
     \ad_{(b,\xi_b)}^* \left( \pder{l}{\xi_b} \right) ,
     (b,\eta_b) \right \rb + \left\lb
    \pder{ \ell }{b} - \frac{D}{Dt} \left( \pder{ \ell }{\dot{b}} \right) +
     i_{\dot{b}}C_{\partial \ell / \partial \xi_b}  ,
     \delta b \right\rb dt}.
  \end{align*}
  As $\eta_b$ and $\delta b$ are arbitrary we arrive at (4).  The
  above argument is reversible, and thus we have proven equivalence of
  (3) and (4) as well.
\end{proof}

  Of course, when $M$ is a flat manifold such as $\mathbb{R}^d$, most people have opted to write the equations of motion for a body immersed in a fluid in the form of equations \eqref{eq:EofM1} and \eqref{eq:EofM2}.
  While the Lagrange-Poincar\'e equations appear different, they assume the same form as the standard equations for fluid-structure interactions when the underlying manifold is $\mathbb{R}^d$, as we now show.

\begin{thm} \label{thm:main2}
Let $M = \mathbb{R}^d$.  Then \eqref{eq:LP_hor} and \eqref{eq:LP_ver} are equivalent to
\begin{align} 
\frac{D}{Dt} \left( \pder{L_{B}}{\dot{b}} \right) -
\pder{L_{B}}{b} &= F_p \label{eq:normal_ppl2} \\
\pder{u}{t} + u \cdot \nabla u &= -\nabla p \label{eq:normal_ppl1} 
\end{align}
 where $u = \xi_b + \sigma(b,\dot{b})$ , $\nabla p$ is a
 Lagrangian parameter which enforces incompressibility, and $F_p \in T^*
 B$ is a force on the boundary of the body given by 
\begin{equation} \label{boundary_force}
\lb F_p , \delta b \rb = - \int_{\partial \bulk_b}{ \lb p(x)
  n(x) , (\delta b \circ b^{-1})(x) \rb d^3x}.
\end{equation}
\end{thm}
\begin{proof}
Assume that the Lagrange-Poincar\'{e} equations \eqref{eq:LP_hor} and \eqref{eq:LP_ver} hold.  Our goal is to show that these equations yield the same dynamics as equations \eqref{eq:normal_ppl1} and \eqref{eq:normal_ppl2}.  This proof is long so we will begin with a roadmap.

{\bf Roadmap:}

\begin{enumerate}
  \item[\bf 1.] We prove that \eqref{eq:LP_ver} is equivalent to \eqref{eq:normal_ppl1}.
  \item[\bf 2.] We note that the horizontal operator
    \[
    \mathcal{LP}_{\hor}( \cdot ) := \frac{d}{dt} \left( \pder{ (\cdot)}{\dot{b}} \right) - \pder{(\cdot)}{b} 
    \]
    is linear on $C^{\infty}(TB \oplus \tilde{\mathfrak{g}})$.  Therefore the right hand side of \eqref{eq:LP_hor} may be written as $\mathcal{LP}_{\hor}(L_{B}) + \mathcal{LP}_{\hor}(\ell_{\mathrm{fluid}}) = C_{\partial \ell / \partial \xi_b}( \dot{b}, \cdot)$.
 \item[\bf 3.] We compute $\mathcal{LP}_{\hor}( \ell_{\mathrm{fluid}})$ as follows:
    \begin{enumerate}
      \item[\bf 3a.] We compute $ \pder{\ell_{\mathrm{fluid}}}{b}$;
      \item[\bf 3b.] We then compute $ \frac{D}{Dt} \left(  \pder{\ell_{\mathrm{fluid}}}{ \dot{b} } \right)$; % <----
      \item[\bf 3c.] We observe that \eqref{eq:LP_hor} can be written as $\mathcal{LP}_{\hor}(\ell_{\mathrm{fluid}}) =  C_{\partial \ell / \partial \xi_b}(\dot{b}, \cdot) - F$ for a judiciously chosen $F$.
      \end{enumerate}
  \item[\bf 4.] The linearity of $\mathcal{LP}_{\mathrm{hor}}$ implies that $\mathcal{LP}_{\hor}(L_{B}) = F$.
  \item[\bf 5.] Lastly, we prove that $F \equiv F_p$, where $F_p$ is the boundary force defined in \eqref{boundary_force}.
\end{enumerate}

%------------- STEP 1
{\bf Step 1.}
From now on we will omit the basepoint in the expressions involving $\sigma$, and  denote $\sigma(b, \dot{b})$ simply by $\sigma(\dot{b})$ (and similarly for $\sigma(\delta b)$).

It is simple to verify that
  \[
   \left\lb \pder{ \ell}{\xi_b}, \delta \xi_b \right\rb = \int_{\bulk_{b}}{ \lb \xi_b + \sigma(\dot{b}) , \delta \xi_b \rb d^3 x } = \int_{\bulk_{b}}{ \lb u , \delta \xi_b \rb d^3 x}.
  \]
  In other words, $\pder{ \ell }{\xi_b} = u^{\flat}$ where we have invoked the Riemannian flat operator.  By equation \eqref{eq:dual_cov_der} we see that
\[
\frac{Du^\flat }{Dt} = \pder{u^\flat}{t} + \ad^*_{\sigma(\dot{b})}( u^\flat).
\]
If we substitute these identities into the vertical LP equation we get
\[
\pder{u^\flat}{t} + \ad^*_{\sigma(\dot{b}) }(u^\flat) = -
\ad^*_{(b, \xi_b) }(u^\flat).
\]
which implies $\pder{u^\flat}{t} + \ad^*_{u} ( u^{\flat}) = 0$. This is Arnold's description of the Euler equation $\pder{u}{t} + u \cdot \nabla u = -\nabla p$ \cite[Section1.5]{ArKh1992}.

%---------------- STEP 2
{\bf Step 2.}
By inspection we can see that $\mathcal{LP}_{\hor}$ is linear and so the horizontal LP equation can be written as
    $\mathcal{LP}(L_{B}) + \mathcal{LP}(\ell_{\mathrm{fluid}}) = C_{\partial \ell / \partial \xi_b}( \dot{b}, \cdot)$,
    where $\ell_{\mathrm{fluid}}$ is the kinetic energy of the fluid, so that $\ell = L_{B} + \ell_{\mathrm{fluid}}$.
   Moreover, note that $\pder{\ell }{b} = \pder{ \ell_{\mathrm{\mathrm{fluid}}}}{b}+ \pder{L_B}{b}$.
   
%-----------------------------STEP 3A
{\bf Step 3a}
  By definition $\pder{\ell_{\mathrm{fluid}}}{b}(b,\dot{b},\xi_b) \in T^*_b B$ is given by
  \[
  	\left\lb \pder{ \ell_{\mathrm{fluid}}}{b}(b,\dot{b},\xi_b) , \delta b \right\rb = \left. \frac{d}{d \epsilon} \right|_{\epsilon = 0} \ell_{\mathrm{fluid}}(b_\epsilon, \dot{b}_{\epsilon}, \xi_{b,\epsilon})
  \]
  where $(b_\epsilon, \dot{b}_\epsilon, \xi_{b,\epsilon})$ is an $\epsilon$-dependent curve in $TB \oplus \tilde{\mathfrak{g}}$ such that $\delta b = \left. \frac{d}{d\epsilon} \right|_{\epsilon = 0} b_\epsilon$ and
  $\frac{D}{D\epsilon}(b_\epsilon, \dot{b}_{\epsilon} , \xi_{b,\epsilon})\big|_{\epsilon=0} = 0$; see Definition~\ref{def:partial_derivatives}.
  By equation \eqref{eq_cov_deriv} this implies
  \begin{align} \label{eq:hor_var_of_xi}
  	\left. \frac{d}{d \epsilon} \right|_{\epsilon = 0} \xi_{b,\epsilon} = [ \sigma( \delta b) , \xi_b].
  \end{align}
  By the Reynolds transport theorem we find that
      \begin{align}
        &\left \lb \pder{\ell_{\mathrm{fluid}}}{b}( b, \dot{b}, \xi) ,\delta b \right \rb = \restr{ \frac{d}{d \epsilon }
        }{\epsilon=0} \left( \frac{1}{2} \int_{\bulk_{b_{\epsilon}} }{ \| \xi_{b,\epsilon} +
            \sigma(\dot{b}_{\epsilon}) \|^2 d^3 x } \right) \nonumber \\
        &\qquad = \frac{1}{2} \int_{\partial \bulk_{b}}{ \| \xi_{b} + \sigma(\dot{b}) \|^2
          \sigma(\delta b) \cdot n \, d^2 x } +
        \frac{1}{2}\int_{\bulk_{b}}{ \left. \frac{d}{d \epsilon} \right|_{\epsilon = 0} \left( \| \xi_{b,\epsilon} + \sigma( \dot{b}_\epsilon) \|^2 \right)
          d^3 x} \label{eq:computation01}.
      \end{align}
      Here we have used the fact that the boundary of $\bulk_{b_{\epsilon}}$ moves with velocity $\sigma(\delta b)$ as we vary $\epsilon$ from $0$.
      We can then focus on the second term in the above sum and apply the covariant derivative on $T^*B \oplus \tilde{\mathfrak{g}}^*$ to find
      \begin{align*}
      	&\frac{1}{2} \left. \frac{d}{d\epsilon} \right|_{\epsilon = 0} \left( \|  \xi_{b,\epsilon}(x) + \sigma( \dot{b}_\epsilon) \|^2 \right) =
		\left\lb \xi_{b}(x) + \sigma(\dot{b})(x) , \left. \frac{d}{d\epsilon} \right|_{\epsilon = 0}( \xi_{b,\epsilon}(x) + \sigma(\dot{b}_{\epsilon})(x) ) \right\rb \\
	&\qquad	=\left\lb \xi_{b}(x) + \sigma(\dot{b})(x) , - [\xi_{b}, \sigma(\delta b) ] + \left. \frac{d}{d\epsilon} \right|_{\epsilon =0}( \sigma(\dot{b}_\epsilon) ) \right\rb \\
	&\qquad	=\left\lb \xi_{b}(x) + \sigma(\dot{b})(x) , - [\xi_{b}, \sigma(\delta b) ] + \left \lb \pder{ \sigma }{ b }(\dot{b}) , \delta b \right \rb + \sigma\left(\frac{D \dot{b} }{D \epsilon} \right) \right\rb
      \end{align*}
      where the second line is obtained from the first line via equation \eqref{eq:hor_var_of_xi} and the third line is gained by viewing $\sigma$ as a vector-valued one-form and invoking proposition \ref{prop:df=dmf+def}.

 Because $\dot{b}_\epsilon$ is a parallel translate of $\dot{b}$ we
 know that $\frac{D \dot{b}_\epsilon}{D\epsilon}=0$.  This allows us
 to drop one term in the last line.  Substitution into equation \eqref{eq:computation01} yields
 \begin{align*}
 	\left \lb \pder{\ell_{\mathrm{fluid}}}{b}( b, \dot{b}, \xi_b) ,\delta b \right \rb =& \frac{1}{2} \int_{\partial \bulk_{b}}{ \| \xi_{b}(x) + \sigma(\dot{b})(x) \|^2
          \sigma(\delta b) \cdot n \, d^2 x } \\
          &+ \int_{\bulk_{b}}{ \left\lb \xi_{b}(x) + \sigma(\dot{b})(x) , -[\xi_{b}, \sigma(\delta b) ] + \left \lb \pder{ \sigma }{ b }(\dot{b}) , \delta b \right \rb \right\rb d^3 x}.
 \end{align*} 
 Finally, we may reduce clutter by substituting $u = \xi_b + \sigma(\dot{b})$ and invoking the divergence theorem to transform the surface integral into a volume integral.  This yields
  \begin{equation} \label{eq:computation00} \begin{split}
   \left \lb \pder{{\ell}_{\mathrm{fluid}}}{b}(b,\dot{b},\xi_b) , \delta b \right \rb = \int_{\bulk_{b}} \Big\{ & \mathrm{div} \left( \frac{1}{2} \| u \|^2 \sigma(\delta b) \right) \\
   &+ \left \lb u , \left[  \sigma(\delta b), \xi_b \right] + \left \lb \pder{\sigma}{b}(\dot{b}),\delta b \right \rb \right \rb \Big\} d^3 x 
 \end{split}\end{equation}
 % ----------------------------------STEP 3B
 {\bf Step 3b}
 We find that the fiber derivative of ${\ell}_{\mathrm{fluid}}$ is given by
\begin{align*}
\left \lb \pder{{\ell}_{\mathrm{fluid}}}{\dot{b}} (b,\dot{b},\xi_b) , \delta b \right \rb &= \restr{\frac{d}{d\epsilon}}{\epsilon = 0}  \ell_{\mathrm{fluid}}(b, \dot{b} + \epsilon \delta b,\xi_b) \\
&= \int_{\bulk_{b}}{
  \left \lb \xi_b + \sigma(\dot{b}) , \sigma(\delta b)
  \right \rb d^3 x}\\
   &=\lb u^\flat , \sigma(\delta b) \rb.
\end{align*}
  By the definition of the covariant derivative on $\tilde{\mathfrak{g}}^*$ we
find
\begin{align}
\left \lb \frac{D}{Dt} \left( \pder{{\ell}_{\mathrm{fluid}}}{\dot{b}} \right) , \delta b
\right \rb &=
\underbrace{\frac{d}{dt}  \left \lb \pder{ {\ell}_{\mathrm{fluid}} }{\dot{b}}, \delta
    b \right \rb
}_{T_1} - \underbrace{ \left \lb \pder{ {\ell}_{\mathrm{fluid}} }{\dot{b}} , \frac{D \delta b}{Dt} \right \rb }_{T_2}. \label{eq:computation02}
\end{align}
We can see by direct computation that
\[
	T_2 = \int_{\bulk_{b}}{ \left\lb u , \sigma \left( \frac{D \delta b}{Dt} \right) \right\rb d^3x}.
\]
This is as far as we need to go with $T_2$ and we will now rework $T_1$.  Firstly, we note that $T_1$ is the time derivative of an integral over a time dependent domain, so we must invoke the Reynolds transport theorem a second time.  This yields the equivalence
\[
	T_1 = \int_{\partial \bulk_b}{ \lb \xi_{b} + \sigma(\dot{b}) , \sigma(\delta b) \rb \lb \xi_b + \sigma(\dot{b}) , n \rb d^2 x} + \int_{\bulk_b}{ \frac{d}{dt}  \lb \xi_b +\sigma(\dot{b}), \sigma(\delta b) \rb } d^3x
\]
By construction $\lb \xi_b , n \rb = 0$ on the boundary of $\bulk_b$ so that
\[
	T_1 = \int_{\partial \bulk_b}{ \lb \xi_{b} + \sigma(\dot{b}) , \sigma(\delta b) \rb \lb \sigma(\dot{b}) , n \rb d^2 x} + \int_{\bulk_b}{ \frac{d}{dt}  \lb u, \sigma(\delta b) \rb \, d^3 x}.
\]
Finally substituting $u = \sigma(\dot{b}) + \xi_b$ yields
\[
	T_1 = \int_{\partial \bulk_b}{ \lb u , \sigma(\delta b) \rb \lb \sigma(\dot{b}), n \rb d^2 x} + \int_{\bulk_{b}}{\left( \left \lb \pder{u}{t} , \sigma(\delta b) \right \rb + \left \lb u , \pder{}{t} \sigma(\delta b) \right \rb \right) d^3 x}.
\]
The term $\pder{}{t} \sigma(\delta b)$ can be handled by invoking proposition \ref{prop:df=dmf+def} a second time.  Explicitly, this give us the equivalence
\[
 \pder{}{t} \sigma(\delta b) = \left\lb \pder{\sigma}{b} (\delta b) , \dot{b} \right\rb + \sigma \left( \frac{D \delta b}{Dt} \right)
\]
where we have used the fact that $\sigma$ is fiberwise linear on $TB$, and therefore $\sigma$ is equal to its own fiber derivative.  We can substitute the above computation into $T_1$ to get the final expression
\begin{multline*}
	T_1 =  \int_{\partial \bulk_b}{ \lb u , \sigma(\delta b) \rb \lb \sigma(\dot{b}), n \rb d^2 x} \\ + \int_{\bulk_{b}}{ \left( \left\lb \pder{u}{t} , \sigma(\delta b) \right\rb + \left\lb u , \left\lb \pder{\sigma}{b} (\delta b) , \dot{b} \right\rb + \sigma \left( \frac{D \delta b}{Dt} \right) \right\rb \right) d^3 x}
\end{multline*}
We can finally revisit equation \eqref{eq:computation02}, and write it as
\begin{align*}
&\left \lb \frac{D}{Dt} \left( \pder{{\ell}_{\mathrm{fluid}}}{\dot{b}} \right) , \delta b \right \rb = \, T_1 - T_2 \\
	&\quad = \int_{\partial \bulk_b}{  \lb u , \sigma(\delta b) \rb \lb \sigma(\dot{b}) , n \rb d^2 x } + \int_{\bulk_b}\left( \left\lb \pder{u}{t} , \sigma(\delta b) \right\rb + \left\lb u , \left \lb \pder{\sigma}{b}(\delta b) , \dot{b} \right \rb \right\rb \right) d^3 x
\end{align*}
We can substitute the vertical equation $\pder{u}{t} = -\nabla p - u \cdot \nabla u$ and (as is customary) we will convert the surface integral into a volume integral via the divergence theorem.  This yields
\begin{multline} \label{eq:computation03} 
	\left \lb \frac{D}{Dt} \left( \pder{\ell_{\mathrm{fluid}}}{\dot{b}} \right), \delta b \right \rb =\\ \int_{\bulk_b} \Big\{  \mathrm{div} \left( \lb u , \sigma(\delta b) \rb \sigma(\dot{b}) \right) 
		- \lb \nabla p + u \cdot \nabla u, \sigma(\delta b) \rb + \left \lb u , \left \lb \pder{\sigma}{b}(\delta b) , \dot{b} \right \rb \right \rb \Big\} d^3 x 
\end{multline}

{\bf Step 3c}
%------------------STEP 3C
By subtracting \eqref{eq:computation00} from \eqref{eq:computation03} we find
\begin{align*}
	\left\lb \frac{D}{Dt}  \left( \pder{\ell_{\mathrm{fluid}}}{\dot{b}} \right) - \pder{\ell_{\mathrm{fluid}}}{b}, \delta b \right \rb=  \int_{\bulk_b} \Big\{ & \mathrm{div} \left( \lb u , \sigma(\delta b) \rb \sigma(\dot{b}) - \frac{1}{2} \|u \|^2 \sigma(\delta b) \right) \\
	&- \lb \nabla p + u \cdot \nabla u , \sigma(\delta b) \rb 
	- \lb u , [\sigma(\delta b) , \xi_b ] \rb \\
	& + \left \lb u , \left \lb \pder{\sigma}{b}(\delta b) , \dot{b} \right \rb - \left \lb \pder{ \sigma}{b}(\dot{b}) , \delta b \right \rb \right \rangle \Big\}  d^3x
\end{align*}
We may now invoke proposition \ref{prop:da(v,w)=da(w)v-da(v)w} to replace the term $\left \lb \pder{\sigma}{b}(\delta b) , \dot{b} \right \rb - \left \lb \pder{ \sigma}{b}(\dot{b}),\delta b\right \rb$ with $d \sigma(\dot{b}, \delta b)$.
Finally we use the equivalence $C(\dot{b}, \delta b) = d\sigma(\dot{b},\delta b) - [\sigma(\dot{b}), \sigma(\delta b) ]$ to rewrite the last equation as
\begin{align*}
	\left\lb \frac{D}{Dt}  \left( \pder{\ell_{\mathrm{fluid}}}{\dot{b}} \right) - \pder{\ell_{\mathrm{fluid}}}{b}, \delta b \right \rb
	 = \int_{\bulk_b} \Big\{ & \mathrm{div} \left( \lb u , \sigma(\delta b) \rb \sigma(\dot{b}) - \frac{1}{2} \|u \|^2 \sigma(\delta b) \right) \\
	&- \lb \nabla p + u \cdot \nabla u , \sigma(\delta b) \rb 
	- \lb u , [\sigma(\delta b) , \xi_b ] \rb \\
	&- \lb u , [\sigma(\delta b) , \sigma(\dot{b}) ] \rb 
	+ \lb u , C(\dot{b},\delta b) \rb \Big\} d^3x
\end{align*}
  We define the term $C_{\partial \ell / \partial \xi_b}(\dot{b},\delta b) = \int_{\bulk_b}{\lb u , C(\dot{b},\delta b) \rb d^3x}$, so that we may conclude 
  \begin{equation} \label{end_of_step_3}
  	\mathcal{LP}_{\hor}( \ell_{\mathrm{fluid}}) = C_{\partial \ell / \partial \xi_b}(\dot{b}, \cdot) - F(b,\dot{b},u).
\end{equation}
for some $F(b,\dot{b},u) \in T^{\ast}B$.

{\bf Step 4.}
%---------------- STEP 4
Given the linearity of $\mathcal{LP}_{\hor}$, \eqref{eq:LP_hor} give us
  \begin{align*}
  	C_{\partial \ell / \partial \xi_b}(\dot{b}, \cdot ) &= \mathcal{LP}_{\hor}( \ell ) = \mathcal{LP}_{\hor}( \ell_{\mathrm{fluid}}) + \mathcal{LP}_{\hor}(L_{B}) \\
	&= C_{\partial \ell / \partial \xi_b}(\dot{b}, \cdot) - F + \mathcal{LP}_{\hor}(L_{B}),
  \end{align*}
  so that $\mathcal{LP}_{\mathrm{hor}}(L_{B}) = F$.
  
  {\bf Step 5.}
  %------------ STEP 5
  To reduce clutter, set $v = \sigma(\delta b)$ and $w = \sigma(\dot{b})$ so that the term given by $F$ may be rewritten as
\[
F = \int_{\bulk_{b}}{ \mathrm{div} \left(  (u \cdot v) w -
    \frac{1}{2} (u \cdot u) v \right) + u \cdot  [v,u]_{JL}  - ( \nabla p
  + u \cdot \nabla u ) \cdot v \, d^3x}
\]
  Using Einstein's summation notation we may write $F$ as:
\[%begin{align*}
F =  \int_{\bulk_{b}}{ \pder{}{x^k} \left( u^iv^iw^k -
      \frac{1}{2}u^iu^iv^k \right) + u^k \left( \pder{u^k}{x^i}v^i -
      \pder{v^k}{x^i}u^i \right) - \left( \pder{p}{x^k}v^k + u^i v^k
      \pder{u^k}{x^i} \right) d^3x}.
\]
Using the fact that $u$ and $v$ are divergence free, so that 
 $\pder{v^k}{x^k} =
 \pder{u^k}{x^k} = 0$, this simplifies to 
\begin{align*}
F &= \int_{\bulk_{b}}{\pder{}{x^k}( v^iu^i) (w^k-u^k) -
   \pder{p}{x^k}v^k d^3x } \\
&= \int_{\bulk_{b}}{ \mathrm{div} \left( - p(x) v(x) -
    \left( v(x) \cdot u(x) \right) \xi_b(x) \right) d^3x} \\
&= \int_{\partial \bulk_{b}}{ \lb - p(x) \sigma(\delta b)(x) -
  \lb \sigma(\delta b) , u(x) \rb \xi_b(x) , n(x) \rb d^2x}.
\end{align*}
Noting that $\xi_b = u -
 \sigma(b,\dot{b}) \equiv u - v$ is tangent to the boundary, and therefore
 orthogonal to the unit normal $n$ we can drop the term
 involving $\xi_b$ so that
\[
F = - \int_{\partial \bulk_b}{ \lb p(x) \sigma(\delta b)(x), n(x)
  \rb d^2x }
\]  
and by the no-slip condition of $\sigma$ we find
\[
F = - \int_{\partial \bulk_{b}}{ \lb p(x) n(x) , (\delta b  \circ b^{-1})(x)  \rb d^2x },
\]
so that $F = F_p$, as claimed.

Thus we may conclude that the Euler-Lagrange equations for $L_{B}$ are given by
\[
  \left\lb \frac{D}{Dt} \left( \pder{L_{B}}{\dot{b}} \right) -
  \pder{L_{B}}{b} , \delta b \right\rb =
  - \int_{\partial \bulk_{b}}{\lb p(x) n(x), (\delta b \circ
    b^{-1})(x) \rb d^2 x}
\]
for arbitrary variations $\delta b$ and the theorem follows.
\end{proof}

%%%%%%%%%%%%%%%%%%%%
\section{Dissipative Lagrangian mechanics in the finite Reynolds regime} \label{sec:Navier_Stokes}
%%%%%%%%%%%%%%%%%%%%

To add a viscosity to our formulation we need to alter our constructions in two ways.
Firstly, we incorporate the no-slip boundary condition directly into our configuration manifold.
We do this by using the submanifold
\[
Q_{\ns} := \{ (b,\varphi) \in Q  \mid b(x) = \varphi( b_{\rm ref}(x ) ) \quad \forall x \in \partial \body \}.
\]
  The group of particle relabeling symmetries is now the group of diffeomorphisms of $\bulk_{\rm ref}$ which keep the boundary point-wise fixed,
  \begin{align*}
  G_{\ns} := \{ \psi \in \SDiff( \bulk_{\rm ref} ) \mid \psi(x) = x \quad \forall x \in \partial \bulk_{\rm ref} \}.
  \end{align*}
  The Lie algebra of $G_{\ns}$ is the set
  \[
  	\mathfrak{g}_{\mathrm{ns}} = \{ \eta \in \mathfrak{X}_{\vol}(\bulk_{\rm ref}) \quad \vert \quad \eta(x) = 0 , \forall x \in \partial \bulk_{\rm ref} \},
  \]
and $P_{\ns} = TQ_{\ns} / G_{\ns}$ changes accordingly: $P_{\ns}$ consists of all triples $(b, \dot{b}, u)$ such that $u(x) = \dot{b}(b^{-1}(x))$ for all $x \in \partial \bulk_b$.
Similarly, the bundle $\tilde{\mathfrak{g}}_{\ns}$ consists of pairs $(b, \xi_b)$ such that $\xi_b \in \mathfrak{X}_{\rm div}(\bulk_b)$ satisfies $\xi_b(x) = 0$ for any $x \in \partial \bulk_b$.
It can be proven, as before, that $P_{\ns}$ is isomorphic to $TB \oplus \tilde{\mathfrak{g}}_{\ns}$ upon choosing a section $\sigma: (b,\dot{b}) \in TB \mapsto \sigma(b,\dot{b}) \in \mathfrak{X}_{\rm div}( \bulk_b)$.
  
  The second change that we need to make is the addition of a viscous friction force. For a viscous fluid with viscosity $\nu > 0$, the viscous force is a map $F_\nu : P_{\ns} \to P_{\ns}^\ast$ defined in Euclidean coordinates by
\begin{align}
	\lb F_{\nu}(b,\delta b,u) , (b,\dot{b},v) \rb & = \nu \int_{ \bulk_{b}}{ \mathrm{trace} ( \nabla u^T \nabla v) \, d^3 x } \nonumber \\
	& = - \nu \int_{\bulk_{b}}{ \lb \Delta u , v \rb \, d^3 x} + \nu \int_{\partial\bulk_{b}} \left< \frac{\partial u}{\partial n}, v \right> \, d^2 x \label{eq:viscous_force}
\end{align}
While this map can also be defined for an arbitrary Riemannian manifold, our prime concern below is with the Euclidean case defined here.
Below, we will also need the pull-back of  $F_\nu$ to $TQ_{\ns}$, denoted by $\tilde{F}_\nu$ and defined by 
\begin{equation} \label{lifted-viscous-force}
  \big\lb \tilde{F}_\nu(b, \dot{b}, \varphi, v_\varphi), (b, \delta b, \varphi, w_\varphi) \big\rb = 
   \left \lb F_\nu (b, \dot{b}, v_\varphi \circ \varphi^{-1}), (b, \delta b, w_\varphi \circ \varphi^{-1}) \right\rb,
\end{equation}
for all $(b, \dot{b}, \varphi, v_\varphi), (b, \delta b, \varphi, w_\varphi) \in TQ_{\ns}$.

Because of the isomorphism between $P_{\ns}$ and $TB \oplus \tilde{\mathfrak{g}}_{\ns}$, we can also view $F_\nu$ as a map from $TB \oplus \tilde{\mathfrak{g}}_{\ns}$ to $T^*B \oplus
  \tilde{\mathfrak{g}}_{\ns}^*$, denoted by the same letter and given by
\[
 \lb F_\nu(b,\dot{b},\xi) , (b,\delta b,\eta) \rb = \nu \int_{\bulk_b}{
   \mathrm{trace} \left( \nabla( \sigma(\dot{b}) + \xi)^T \nabla
     ( \sigma(\delta b) + \eta) \right) \, d^3x }.
\]
%Moreover, $F_\nu$ can be split as the direct sum of two parts, $F_\nu
%= F_{\ns} \oplus f_{\nu}$, where $F_{\mathrm{ns}}: TB \oplus
%\tilde{\mathfrak{g}}_{\ns} \to T^{\ast} B$ is the boundary force which
%will enforce the no-slip condition and $f_{\nu}: TB \oplus
%\tilde{\mathfrak{g}}_{\ns} \to \tilde{\mathfrak{g}}_{\ns}^{\ast}$ is the viscous force on the
%fluid.  Finally, the force $F_\nu$ can be lifted to a force $\tilde{F}_\nu:TQ_{\ns} \to T^{\ast}Q_{\ns}$ in the obvious way.
%These forces allows us to derive the equations of motion for a body in a Navier-Stokes fluid as shown in the theorem below.

\begin{thm} \label{thm:main2_viscous}
Let $L$ be the fluid-structure Lagrangian of equation \eqref{fsi_lagrangian} restricted
to $TQ_{\ns}$ and $\ell: TB \oplus
\tilde{\mathfrak{g}}_{\ns} \to \mathbb{R}$ the reduced
Lagrangian written in terms of a section, $\sigma:(b,\dot{b}) \in TB \to \sigma(b,\dot{b}) \in \mathfrak{X}_{\rm div}(\bulk_b)$.  Given a curve
$q(t) = (b(t), \varphi(t)) \in Q_{\ns}$, let $(b, \dot{b} , \xi_b)_t$ be the induced curve in $TB \oplus \tilde{\mathfrak{g}}_{\ns}$, i.e. such that $\xi_b = \dot{\varphi} \circ \varphi^{-1} - \sigma(b,\dot{b})$.
  Then the following are equivalent:
\begin{enumerate}
\item The curve, $q(\cdot)$, satisfies the variational principle
\[
	\delta \int_0^t L(q,\dot{q}) dt = \int_0^t \langle \tilde{F}_\nu(q,\dot{q}) , \delta q \rangle dt
\]
with respect to variations of $q( \cdot )$ with fixed end points.
\item The curve, $q(\cdot )$, satisfies the externally forced Euler-Lagrange equations
\[
  \frac{D}{Dt} \left( \pder{L}{\dot{q}} \right) - \pder{L}{q} = \tilde{F}_\nu,
\]
where $\tilde{F}_\nu$ is the lift of the viscous force $F_\nu$,
defined in \eqref{lifted-viscous-force}.
\item The curve $(b,\dot{b}, \xi_b)( \cdot )$ satisfies the variational principle
\[
	\delta \int_0^t \ell( b , \dot{b} , \xi_b) dt = \int_0^t \langle F_\nu( b,\dot{b},\xi_b) , \eta_b \rangle dt
\]
with respect to arbitrary variations of the curve $b(t)$ with fixed end-points and covariant variations $\delta^\sigma \xi_b$
 given by \eqref{eq:vert_variations} for an arbitrary curve $(b,\eta_b)(t) \in \tilde{\mathfrak{g}}_{\ns}$ with vanishing end-points.

\item  The curve $(b, \dot{b} , \xi_b)( \cdot )$ satisfies the externally forced Lagrange-Poincar\'{e} equations
  \begin{align*}
    \frac{D}{Dt} \left( \pder{\ell}{\dot{b}} \right) - \pder{l}{b} =
    i_{\dot{b}} C_{\partial \ell / \partial \xi_b}  + F_{\ns} \\
    \frac{D}{Dt} \left( \pder{\ell}{\xi_b} \right) = -\ad_{(b,\xi_b)}^* \left(
      \pder{\ell}{\xi_b} \right) + f_{\nu}
  \end{align*}
where,
\begin{align*}
  \lb F_{\ns}(v_b, \xi_b) , \delta b \rb &= \nu \int_{\partial \bulk_b}{
  \lb n(x) \cdot \nabla u(x)  , \delta b ( b^{-1}(x) ) \rb \, d^2x} \quad \forall \delta b \in T_b B \\
  \lb f_{\nu}(v_b, \xi_b) , \eta_b \rb &= - \nu \int_{\bulk_b}{ \lb \Delta (
    \sigma(v_b) + \xi_b ) , \eta_b \rb \, d^3x } \quad \forall \eta_b \in \tilde{\mathfrak{g}}_{\ns}.
\end{align*}
\item  If $M = \mathbb{R}^d$ then the curve $(b, \dot{b}, \xi_b)(\cdot)$ satisfies
\begin{align*}
  \pder{u}{t} + u \cdot \nabla u &= -\nabla p + \nu \Delta u \\
  \frac{D}{Dt} \left( \pder{L_{B}}{\dot{b}} \right) -
  \pder{L_{B}}{b} &= F_p + F_{\mathrm{ns}}
\end{align*}
where $u = \xi_b + \sigma(b,\dot{b})$, and $F_p$ is the pressure force on the boundary of the body given in
Theorem \ref{thm:main1}.
\end{enumerate}
\end{thm}

\begin{proof}
  The equivalence of (1) and (2) is standard.
  The equivalent of (1) and (3) is proven using the same manipulations described in Theorem \ref{thm:main1}.
  The equivalence of (2) and (4) is merely a statement of the
  Lagrange-Poincar\'{e} reduction theorem with the addition of a
  force.  Assume (4) and we will prove equivalence with (5).  We
  will deal with the vertical equation first by
  understanding the force on the fluid.   By
  \eqref{eq:viscous_force} the total viscous force is
 \begin{multline*}
   \lb F_\nu(b,\dot{b},\xi_b) , (b, \delta b, \eta_b ) \rb = \\
    - \nu \int_{\bulk_{b}}{ \lb \Delta u , \eta_b + \sigma(\delta b) \rb \, d^3 x} + \nu \int_{\partial\bulk_{b}} \left< \frac{\partial u}{\partial n}, \delta b (b^{-1}(x)) \right> \, d^2 x, 
\end{multline*}
where $v$ of $\eqref{eq:viscous_force}$ is given by $v = \sigma(\delta{b}) + \eta_b$ and we have invoked the boundary conditions $\eta_b = 0$ and $\sigma(\delta b) = \delta b \circ b^{-1}$ on $\partial \bulk_b$.
 
Following the same procedure as for the inviscid case (see ``step 1'' in the proof of Theorem \ref{thm:main2}) we arrive at the equation
\[
\dot{u}^\flat + \ad^*_{u} ( u^{\flat}) = \nu \Delta u^\flat.
\]
on $\bulk_b$.  On $\mathbb{R}^d$ this becomes
\[
\pder{u}{t} + u \cdot \nabla u = - \nabla p + \nu \Delta u.
\]
  Additionally, if we denote the horizontal Lagrange-Poincar\'{e} operator
\[
 \mathcal{LP}_{\hor} := \frac{D}{Dt} \left( \pder{}{\dot{b}} \right) - \pder{}{b}
\]
we find that (just as in the inviscid case), $\mathcal{LP}_{\hor}( \ell_{\mathrm{fluid}} )
= i_{\dot{b}} C_{\partial \ell / \partial \xi_b} - F_{p}$, where
\[
\lb F_{p}(b,\dot{b}) , w_b \rb = - \int_{\partial \bulk_b }{ p(x) \lb
n(x) , w_b ( b^{-1}(x) ) \rb d^2x}
\] 
is the pressure force on the body.
Note that $\ell = L_{B} + \ell_{\mathrm{fluid}}$ and $\mathcal{LP}_{\hor}$ is a
linear operator on the set of reduced Lagrangians.  Therefore the (forced) horizontal Lagrange-Poincare equation states
\[
  \mathcal{LP}_{\hor}(L_{B}) + \mathcal{LP}_{\hor}( \ell_{\mathrm{fluid}} ) =
  i_{\dot{b}} C_{\partial \ell / \partial \xi_b} + F_{\mathrm{ns}}.
\]
  Putting all these pieces together yields the
forced Euler-Lagrange equation on $B$ given by
\[
\frac{D}{Dt} \left( \pder{L_{B}}{\dot{b} } \right) -
\pder{L_{B}}{b} = F_p + F_{\mathrm{ns}}.
\]
Reversing these steps proves the equivalence of (5) with (4).
\end{proof}

%%%%%%%%%%%%%%%%%%%%
\section{Conclusion} \label{sec:conclusion}
%%%%%%%%%%%%%%%%%%%%

In this paper we have identified the configuration manifold of fluid structure interactions
as a principal $G$-bundle and found equations of motion on $P = TQ/G$
which are equivalent to the standard ones obtained in continuum mechanics literature \cite{Batchelor}.
We were then able to derive the equations of motion for a body immersed in a Navier-Stokes fluid by
adding a viscous force and a no-slip condition.
Given this geometric perspective we can extend, specialize, and apply our
findings.  Such endeavors entail future work.
  
{\bf Rigid bodies:}
  We could restrict the groupoid $\mathcal{G}$ of remark \ref{rmk:groupoid} by considering a subgroupoid
  \[
  	\mathcal{G}_{\mathrm{rigid}} := \{ g \in \mathcal{G} \quad \vert \quad \exists A \in \SE(3) , \beta(g) = A \cdot \alpha(g) \}.
  \]
  We expect the resulting equations of motion to be that of a rigid body immersed in an ideal fluid.
  This would be an interesting generalization of Arnold's discovery that the Euler equations are right trivialized geodesic equations on a Lie group, in that we would have found that the equations for a rigid body in an ideal fluid are right trivialized geodesic equations on a Lie groupoid.
  The equations of motion for a rigid body in a fluid have been treated from a geometric point of view in \cite{Radford_thesis, Kanso2005, Vankerschaver2009, VaKaMa2010},
  but so far the relevance of Lie groupoids in this context has not been explored.
  
{\bf Complex fluids:}
    There exists a unifying framework for understanding complex fluids using the notion of advected parameters \cite{GayBalmaz2009}.
    This framework uses the interpretation of an incompressible fluid as an ODE on the set of volume preserving diffeomorphisms and extends the ideal fluid case by appending parameters in a vector space which are advected by the diffeomorphisms of the fluid.  This framework captures magneto-hydrodynamics, micro-stretch liquid crystals, and Yang-Mills fluids.
    Merging these ideas with the Lagrange-Poincar\'e formalism presented here should be possible.
    Moreover, the notion of ``advection'' fits naturally in the groupoid theoretic framework of remark \ref{rmk:groupoid} where the source map stores the initial state of the system.
    
{\bf Numerical algorithms:}
  Discrete Lagrangian mechanics is a framework for the construction of
numerical integrators for Lagrangian mechanical systems. These
integrators typically preserve some of the underlying geometry, and as
a consequence have good long-term conservation and stability
properties (see \cite{MaWe2001, Hairer2002}).  The ideas of discrete
mechanics have been extended to deal with mechanical systems on Lie
groups (see \cite{LeeLeok08} and the references therein) and in
\cite{Gawlik2011} a variational integrator for ideal fluid dynamics
was proposed by approximating $\SDiff(M)$ with a finite dimensional
Lie group.  It would be interesting to pursue a melding of these
integration schemes by ``discretizing'' the Lie groupoid
$G_{\mathrm{rigid}}$ to obtain a variational integrator for the
interaction of an ideal fluid with a rigid body by approximating it as
a finite dimensional Lie groupoid. The groupoid framework of this
paper would provide a natural inroad into this problem, the more so
since in \cite{Weinstein1995} it was shown that many variational
integrators can viewed as discrete mechanical systems on Lie
groupoids.

\section{Acknowledgements}
The initial background material for the problem was provided to the first author by Jerrold E. Marsden who continued to serve as a kind and
caring mentor even in death.
Many of the ideas in this paper were clarified through conversations with Alan Weinstein who helped the authors get acquainted with the groupoid literature.
Additionally, the authors would like to thank Fran\c{c}ois Gay-Balmaz, Mathieu Desbrun, Darryl D. Holm, and Tudor S. Ratiu for encouragement and clarifications on fluid mechanics and Lagrange-Poincar\'{e} reduction.
Both authors gratefully acknowledge partial support by the European Research Council Advanced Grant 267382 FCCA.  H.~J. is partially supported by NSF grant CCF-1011944.  J. V. is on leave from a Postdoctoral Fellowship of the Research Foundation--Flanders (FWO-Vlaanderen).
This work is supported by the \textsc{irses} project \textsc{geomech} (nr. 246981) within the 7th European Community Framework Programme.

%%%%%%%%%%%%%%%%%%%%%%
\appendix
%%%%%%%%%%%%%%%%%%%%%%

%%%%%%%%%%%%%%%%%%%%
\section{Connections on vector bundles} \label{sec:connections_vb}
%%%%%%%%%%%%%%%%%%%%
In this section we will derive identities from the theory of connections on vector bundles.
These results are mostly used in the proofs of Theorems \ref{thm:main1} and \ref{thm:main2}.

%%%%%%%%%%%%%%%%%%%%
\subsection{Vertical bundles}
%%%%%%%%%%%%%%%%%%%%
  For any vector bundle, $\pi:E \to M$, we may invoke the
  tangent bundle, $\tau_{E}:TE \to E$, and define the vertical bundle,
  $\pi_V : V(E) \to E$, where $V(E) = \mathrm{kernel}( T \pi
)$ and $\pi_V := \restr{\tau_{E}}{V(E)}$.  Equivalently, we may
define the vertical bundle as the set of tangent vectors of $E$ of
the form $\restr{\frac{d}{d\epsilon}}{\epsilon=0}( e + \epsilon f)$
where $f,e$ are in the same fiber.  This gives us the following proposition.

\begin{prop} \label{prop:vertical_lift}
  The map $v^{\uparrow}: (e_1,e_2) \in E \oplus E \mapsto  \restr{ \frac{d}{d \epsilon} }{\epsilon = 0}
\left( e_1 + \epsilon e_2 \right) \in V(E)$ is an isomorphism.
\end{prop}
We call the map $v^\uparrow$ of Proposition \ref{prop:vertical_lift} the \emph{vertical lift}.

%%%%%%%%%%%%%%%%%%%%
\subsection{Connections} \label{sec:connections}
%%%%%%%%%%%%%%%%%%%%
There are three types of connections used in this paper:
Levi-Civita connections, connections on a vector bundle, and principal connections. In this section, we will review the first two notions while principal connections will be discussed in Appendix~\ref{app:connections}. In particular, we define the covariant derivative induced by a vector bundle connection, as this concept will be used to formulate the equations of motion throughout the paper; see for instance Theorems \ref{thm:main1} and \ref{thm:main2}.

\begin{definition}\label{def:connection}
  A \emph{horizontal sub-bundle} on a vector bundle, $\pi:E \to M$, is a
  subbundle, $H(E) \subset TE$, such that $TE = H(E) \oplus
  V(E)$. Given a horizontal sub-bundle we define the projections
  $\mathrm{ver}: TE \to V(E)$ and $\mathrm{hor}: TE \to H(E)$.  The
  projection, $\mathrm{ver}$, is called a \emph{connection}.
\end{definition}

This definition of connection appears in \cite{KMS99} where it is
called a ``generalized connection''.  The choice of a connection has a
number of consequences, the most important being the existence of a 
horizontal lift.

\begin{prop} \label{prop:horizontal_lift}
  Given a connection on a vector bundle $\pi: E \to M$,
  there exists an inverse to the map $\restr{ \tau_E \oplus T\pi
  }{H(E)}  : H(E) \to E \oplus TM$.  We denote this inverse by
  $h^\uparrow: E \oplus TM \to H(E)$, and call it the \emph{horizontal
    lift}.
\end{prop}

\begin{proof}
  Choose an $e \in E$ and let $H_e(E), V_e(E), T_eE$ denote the fibers of the
  horizontal bundle, the vertical bundle, and the tangent bundle above
  $e$.  As $T_eE = H_e(E) \oplus V_e(E)$ and $V_e(E) =
  \mathrm{kernel}(T_e\pi)$, we see that $\restr{T_e\pi}{H_e(E)}$ is
  injective.  Additionally, $\restr{T_e \pi}{H_e(E)}$ must also be
  surjective since $T_e\pi(T_eE) = T_mM$ where $m = \pi(e)$.  Thus
  $\restr{T_e \pi}{H_e(E)}$ is invertible.  We define $h^\uparrow: E
  \oplus TM \to H(E)$ by $h^\uparrow( e , \dot{m} ) =
  \restr{T_{e}\pi}{H_e(E)}^{-1}( \dot{m} )$.
  \end{proof}

We will occasionally use the notation $h^\uparrow_e( \dot{m} ) \equiv
h^\uparrow(e,\dot{m})$.  Additionally, the choice of a connection induces a covariant derivative.

\begin{definition}\label{def:cov_der}
  We define the \emph{vertical drop} to be the map, $v_{\downarrow} : V(E) \to
E$ given by
\[
v_{\downarrow} = \pi_2 \circ (v^{\uparrow})^{-1}
\]
where $\pi_2 : E \oplus E \to E$ is the projection onto the second $E$
component.  Let $I \subset \mathbb{R}$ be an interval so that we may
consider the set $C(I,E)$ consisting of curves from $I$ into $E$.
We define the \emph{covariant derivative} with respect to the
connection $H(E)$ to be the map, $\frac{D}{Dt} : C( I , E) \to E$,
given by
\[
  \frac{D e }{Dt} = v_{\downarrow} \left[ \mathrm{ver} \left(
    \frac{de}{dt} \right) \right]
\]
for a curve $e(\cdot) \in C(I,E)$.
\end{definition}
Lastly, the choice of a connection also provides two partial
derivatives on a vector bundle.
 
\begin{definition} \label{def:partial_derivatives}
  Let $\pi:E \to M$ be a vector bundle and let $f:E \to \mathbb{R}$ be
  a function.  Given a connection we define the partial derivative of $f$ with
  respect to $m$ as the vector bundle morphism, $\pder{f}{m} : E \to T^*M$,
  given by
\[
 \left \lb \pder{f}{m}(e) , \delta m \right \rb := \lb df(e) , h^\uparrow_e(
 \delta m ) \rb;
\]
see \cite{MTA}.
  Additionally we define the \emph{fiber derivative}, $\pder{f}{e}:E
  \to E^*$, given by
\[
\left \lb \pder{f}{e}(e) , e' \right \rb = \lb df(e) , v^\uparrow( e , e')
\rb := \restr{ \frac{d}{d\epsilon} }{\epsilon=0}\left( f( e + \epsilon e')\right).
\]
\end{definition}

These partial derivatives allow us to write down the differential
of a real-valued function in a new way.

\begin{prop} \label{prop:df=dmf+def}
  Let $f: E \to \mathbb{R}$.  Let $\delta e \in TE$ and $\delta m =
  T\pi( \delta e )$.  Then
\[
\lb df(e) , \delta e \rb = \left \lb \pder{f}{m}(e) , \delta m \right \rb + \left \lb \pder{f}{e}(e) , \frac{De(t)}{Dt} \right \rb.
\]
where $e(t)$ is a path in $E$ tangent to $\delta e$.
\end{prop}

\begin{proof}
We observe that
\[
\lb df , \delta e \rb = \lb df , \mathrm{hor}(\delta e) \rb + \lb df ,
\mathrm{ver}(\delta e) \rb.
\]
  By definition, the first term is identically $\lb df ,
  h^{\uparrow}( T\pi(\delta e) ) \rb = \left \lb \pder{f}{m} , \delta m \right \rb$.
  The second term is 
\[
\left \lb df , \mathrm{ver}(\delta e) \right \rb = \left \lb df ,
  v^{\uparrow} \left( e , \frac{De(t)}{Dt} \right) \right \rb = \left \lb \pder{f}{e} ,
  \frac{De(t)}{Dt} \right \rb
\]
\end{proof}

The covariant derivative on curves in $E$ naturally induces a
covariant derivative on curves in $E^*$ given by the condition
\begin{align}
\frac{d}{dt} \lb \alpha , e \rb = \left \lb \frac{D \alpha }{Dt} , e \right \rb +
\left \lb \alpha , \frac{De}{Dt} \right \rb. \label{eq:cov_der2}
\end{align}

\begin{remark} Given vector bundles $E_1$ and $E_2$ with
horizontal sub-bundles $H(E_1)$ and $H(E_2)$ we can define a
horizontal sub-bundle by taking the direct sum: $H(E_1\oplus E_2) = H(E_1) \oplus H(E_2)$.  This
will be valuable later when dealing with equations of motion on the
direct sum of a Riemannian manifold and an adjoint bundle.\qedremark
\end{remark}

%%%%%%%%%%%%%%%%%%%%
\subsection{Connections on Riemannian manifolds}
%%%%%%%%%%%%%%%%%%%%

The two most important examples of connections (for us) are the
Levi-Civita connection of a Riemannian manifold, and a principal connection of a principal
bundle.  We will review these next.  To do this we introduce the
notion of an affine connection.

\begin{definition}
  An affine connection on a vector bundle $\pi:E \to M$ is a mapping,
  $\nabla: \mathfrak{X}(M) \oplus \Gamma(E) \to \Gamma(E)$ such that
  \begin{enumerate}
  \item $\nabla_{f X+g Y}(\sigma) = f \nabla_X \sigma + g \nabla_Y
    \sigma$ for $f,g \in C^{\infty}(M), X,Y \in \mathfrak{X}(M),
    \sigma \in \Gamma(E)$.
   \item $\nabla_{X}(\sigma_1 + \sigma_2) = \nabla_X\sigma_1 +
      \nabla_X \sigma_2$ for $\sigma_1,\sigma_2 \in \Gamma(E)$.
   \item $\nabla_{X}( f \sigma) = X[f] \sigma + f \nabla_X{\sigma}$
     for $f \in C^{\infty}(M), X \in \mathfrak{X}(M), \sigma \in \Gamma(E)$.
  \end{enumerate}
\end{definition}

Let $M$ be a Riemannian manifold with a metric $\lb \cdot , \cdot
\rb_M$.  A connection on $TM$ is a subspace, $H(TM) \subset TTM$ and
the covariant derivative is a map $\frac{D}{Dt}: C(I, TM) \to TM$.  The
covariant derivative uniquely defines an affine connection by the
equivalence $( \nabla_{X}(Y) )(m_0) = \frac{Dv}{Dt}$ where $v(t) =
Y(m(t))$ and $m(t)$ is the integral curve of $X$ with the initial
condition $m_0$ \cite[Section 2.7]{FOM}.  There is a one-to-one
correspondence between affine connections and covariant derivatives.

On a Riemannian manifold there is a
canonical affine connection which satisfies the following two extra properties:
\begin{enumerate}
\item $\nabla_X Y - \nabla_Y X = [X,Y]$ (torsion free).
\item $\lb \nabla_Z X , Y \rb_M + \lb X , \nabla_Z Y \rb_M = Z[ \lb X
  , Y \rb_M]$ (metric)
\end{enumerate}

for any $X,Y,Z \in \mathfrak{X}(M)$.  We call this unique affine connection
the Levi-Civita connection.

\begin{prop} \label{prop:da(v,w)=da(w)v-da(v)w}
  Let $\alpha \in \Omega^1(M)$.  A torsion-free connection (such as the Levi-Civita connection) on $TM$ induces
  the equivalence
\[
d\alpha( v , w ) = \left \lb \pder{\alpha}{x}(w) , v \right \rb - \left \lb
\pder{\alpha}{x}(v) , w \right \rb,
\]
where we are viewing $\alpha$ as an element of $C^{\infty}(TM)$ on the right hand side, and an element of $\Omega^1(M)$ on the left hand side.
\end{prop}

\begin{proof}
	Let $X, Y \in \mathfrak{X}(M)$.  By the definition of $\pder{\alpha}{x}$ viewed as a function on $TM$ we find
	\[
		\left \lb \pder{\alpha}{x}(X) , Y \right \rb (x) = \lb d_{TM} \alpha , h_{X(x)}^{\uparrow}( Y(x) ) \rb,
	\]
	where $d_{TM}$ is the exterior derivative on $TM$.   Note that $h_{X(x)}^{\uparrow}(Y(x))$ is a vector in $TTM$ over $X(x) \in TM$ obtained by parallel transporting the vector $Y(x)$ over a curve with velocity $X(x)$.  In other words
	\[
		\left \lb \pder{\alpha}{x}(X) , Y \right \rb (x) = \restr{ \frac{d}{dt} }{t=0}  \lb \alpha( \tilde{x}(t) ) , Y_t \rb 
	\]
	where $\tilde{x}(t)$ is an arbitrary curve such that $\tilde{x}(0) = x$ and $\restr{ \frac{d\tilde{x}}{dt} }{t=0} = X(x)$.  By the induced covariant derivative on curves of covectors we find
	\[
		\left \lb \pder{\alpha}{x}(X) , Y \right \rb (x) = \left\lb \restr{ \frac{D \alpha(\tilde{x}(t))}{Dt} }{t=0} , Y_0 \right\rb +  \left\lb \alpha(x) , \restr{\frac{DY_t}{Dt}}{t=0} \right\rb.
	\]
	However, as the $Y_t$ is parallel along $\tilde{x}(t)$ the second term is $0$.  Therefore
	\[
		\left \lb \pder{\alpha}{x}(X) , Y \right \rb (x) = \left\lb \restr{ \frac{D \alpha(\tilde{x}(t))}{Dt} }{t=0} , Y_0 \right \rb.
	\]
	As $Y_0 = Y(x)$ we see that we can use the identity $\frac{d}{dt} \left \lb \alpha, Y \right \rb = \left \lb \frac{D \alpha}{Dt} , Y \right \rb + \left \lb \alpha , \frac{DY}{Dt} \right \rb$ to rewrite the above equations as
	\[
		\left \lb \pder{\alpha}{x}(X),Y \right \rb(x) = \restr{ \frac{d}{dt} }{t=0}  \left\lb \alpha(x(t)), Y(x(t)) \right\rb - \left\lb \alpha , \restr{\frac{D}{Dt}}{t=0}(Y(x(t))) \right\rb.
	\]
	As $x \in M$ is arbitrary, we can write this in terms of the affine connection $\nabla$ as
	\[
		\left \lb \pder{\alpha}{x}(X),Y \right \rb = \mathcal{L}_X  \lb \alpha , Y \rb - \lb \alpha , \nabla_X Y \rb.
	\]
	If we swap $X$ and $Y$ and subtract these equations we find
	\[
		\left \lb \pder{\alpha}{x}(X) , Y \right \rb - \left\lb \pder{\alpha}{x} (Y) , X \right\rb =  \mathcal{L}_X \lb \alpha , Y \rb -  \mathcal{L}_Y \lb \alpha , X \rb - \lb \alpha , \nabla_X Y - \nabla_Y X \rb.
	\]
	By the torsion free property,
	\[
		\left \lb \pder{\alpha}{x}(X) , Y \right \rb - \left\lb \pder{\alpha}{x} (Y) , X \right\rb =  \mathcal{L}_X \lb \alpha , Y \rb -  \mathcal{L}_Y \lb \alpha , X \rb - \lb \alpha , [X,Y] \rb.
	\]
	However the right hand side is simply $d\alpha(X,Y)$.
\end{proof}
  Note that Proposition \ref{prop:da(v,w)=da(w)v-da(v)w} presupposes that the exterior derivative is well defined.  This is never a problem on finite-dimensional manifolds.  However, infinite-dimensional manifolds require that we deal with some functional analytic
  concerns which we would like to skip \cite[see Section 7.4]{MTA}.  One can
  verify that the above theorem is true in the infinite-dimensional
  case if the exterior derivative exists.  So we must add this
  existence assumption in the infinite dimensional case.

  Given a vector bundle $\pi : E \to M$ we may define the set of
  $E$-valued forms $\Omega^1(M;E) := \Gamma(E) \otimes\Omega^1(M)$.
  This allows us to define the exterior differential of an $E$ valued
  form by $d ( e \otimes \alpha) = e \otimes d\alpha$ for an arbitrary
  $e \in \Gamma(E)$ and $\alpha \in \Omega^1(M)$.  It is simple to see
  that using this definition for the differential of an $E$-valued
  form Proposition \ref{prop:da(v,w)=da(w)v-da(v)w} holds for
  $E$-valued forms as well.

\section{Principal connections} \label{app:connections}
%%%%%%%%%%%%%%%%%%%%%%
In this appendix we show that the choice of a section, $\sigma : TB\to P$, of the anchor map, $\rho: P \to T B$, defines a principal connection on the principal $G$-bundle $\pi^Q_B: Q \to B$.
The material described here relies on \cite{KoNo1963}; see also \cite{CeMaRa2001}.
This manifestation of a section of the anchor map appeared in early work on particles immersed in fluids where such an object was called an ``interpolation method''   \cite{JaRaDe2012}.
For the sake of simplicity, we treat the case of an inviscid fluid; the extension to viscous flows is straightforward.

\subsection{The connection one-form}

Given a section $\sigma$, we define a $\mathfrak{g}$-valued one-form $A:TQ \to \mathfrak{g}$ by putting 
\begin{equation} \label{eq:conn_one_form_def}
	A(b, \dot{b}, \varphi, \dot{\varphi}) = T\varphi^{-1} \circ \dot{\varphi} - 
		\varphi^\ast [ \sigma(b,\dot{b}) ],
\end{equation}
 where $\varphi^{\ast} [\sigma(b,\dot{b})]$ is the pull-back of the vector field $\sigma(b, \dot{b} ) \in \mathfrak{X}( \bulk_b)$ by the map $\varphi: \bulk_{\rm ref} \to \bulk_{b}$.
To verify that this produces an element of $\mathfrak{g}$ note that $A$ can be written as 
\begin{equation} \label{conn_one_form_interpolation}
	A(b, \dot{b}, \varphi, \dot{\varphi}) = 	\varphi^\ast \left( \dot{\varphi} \circ \varphi^{-1} - \sigma(b,\dot{b}) \right), 
\end{equation}
where the expression between parentheses on the right-hand side is an element of $\mathfrak{X}_{\rm div}( \bulk_b)$.
Since $\varphi$ is volume-preserving and maps the boundary of $\bulk_{b_0}$ to that of $\bulk_{b}$, we have that the pullback vector field in $\mathfrak{X}_{\rm div}(\bulk_b)$ is an element of $\mathfrak{g} = \mathfrak{X}_{\rm div}( \bulk_{\rm ref})$.  We may also write this correspondence as 
\begin{equation} \label{MC_expression}
	A(b, \dot{b}, \varphi,\dot{\varphi}) = \varphi^\ast \left( \Theta_R(\varphi, \dot{\varphi}) - \sigma(b,\dot{b}) \right), 
\end{equation}
where $\Theta_R$ is the right Maurer-Cartan form, defined by $\Theta_R(\varphi, \dot{\varphi}) = \dot{\varphi} \circ \varphi^{-1}$.
While this expression is a bit more involved than the original definition, it will make the computation of the curvature of $A$ below easier.
Strictly speaking, $\Theta_R$ is not a Maurer-Cartan form in the usual sense, but rather lives on the Atiyah groupoid associated to the principal bundle $Q$.
It turns out, however, that $\Theta_R$ has all the properties required of it --- in particular, $\Theta_R$ has zero curvature --- so that we will not dwell on this point any further.
For more information about groupoid Maurer-Cartan forms, see \cite{FeSt2008}.

\begin{prop}
The one-form $A$ defined in \eqref{eq:conn_one_form_def} is a connection one-form.
\end{prop}
\begin{proof}
We need to check two properties:
\begin{enumerate}
\item $A$ is (right) equivariant under the action of $G$ on $Q$: for $\psi \in G$ and $(b, \varphi, \dot{b}, \dot{\varphi}) \in Q$, we have 
\begin{align*}
	A(b, \varphi \circ \psi) \cdot (\dot{b}, \dot{\varphi} \circ \psi) 
		& = T (\varphi \circ \psi)^{-1} \circ \dot{\varphi} \circ \psi 
			- (\varphi \circ \psi)^\ast [ \sigma(b,\dot{b}) ] \\
		& = \psi^\ast ( T \varphi^{-1} \circ \dot{\varphi} - \sigma(b,\dot{b}) ) \\
		& = \mathrm{Ad}_{\psi^{-1}} \, [A(b, \varphi) \cdot (\dot{b}, \dot{\varphi})],
\end{align*}
where $\mathrm{Ad}_{\psi^{-1}} : \mathfrak{g} \to \mathfrak{g}$ is the adjoint action given by $\Ad_{\psi} ( \xi) = T\psi \circ \xi \circ \psi^{-1}$.

\item $A$ maps the infinitesimal generator $\xi_{Q}$ of an element $\xi \in \mathfrak{g}$ back to $\xi$. Recall that the infinitesimal generator is the vector field $\xi_{Q}$ on $Q$ given by 
\[
	\xi_{Q}(b, \varphi) = (0, T \varphi \circ \xi) \in T_{(b, \varphi)} Q,  
\]
for all $\xi \in \mathfrak{g}$ to $\xi$, so that 
\[
	A(b, \varphi) \cdot \xi_{Q}
		= T\varphi^{-1} \circ T\varphi \circ \xi - \sigma(0) = \xi. \qedhere
\]
\end{enumerate}
\end{proof}

Conversely, given a connection one-form $A$, we may define a section of the anchor map $\sigma$ as follows. For $(b, \dot{b}) \in TB$, choose $\varphi$ and $\dot{\varphi}$ to be compatible with $(b, \dot{b})$, that is, so that $(b, \varphi; \dot{b}, \dot{\varphi}) \in TQ$.  We then put 
\begin{equation} \label{eq:induced_interp_method}
	\sigma(b,\dot{b}) = \dot{\varphi} \circ \varphi^{-1} - \varphi_\ast [ A(b, \varphi) \cdot (\dot{b}, \dot{\varphi}) ].
\end{equation}
We first check that this prescription is well-defined, i.e. that $\sigma(b)$ does not depend on the choice of $\varphi$ and $\dot{\varphi}$. Let $\varphi'$ and $\dot{\varphi}'$ be different elements so that $(b, \varphi'; \dot{b}, \dot{\varphi}') \in TQ$. There then exists a diffeomorphism $\psi \in G$ and a vector field $\xi \in \mathfrak{g}$ so that $\varphi' = \varphi \circ \psi$ and $\dot{\varphi}' = \dot{\varphi} \circ \psi + T (\varphi \circ \psi) \circ \xi$. After some calculations, we then have that 
\[
	\dot{\varphi}' \circ \varphi'^{-1} = \dot{\varphi} \circ \varphi^{-1} + (\varphi \circ \psi)_\ast \xi 
\]
and
\[
	\varphi'_\ast [ A(b, \varphi') \cdot (\dot{b}, \dot{\varphi}') ] 
	= \varphi_\ast [A(b, \varphi) \cdot (\dot{b}, \dot{\varphi}) ] + (\varphi \circ \psi)_\ast \xi, 
\]
so that 
\[
	\dot{\varphi} \circ \varphi^{-1} - \varphi_\ast [ A(b, \varphi; \dot{b}, \dot{\varphi}) ]
 	= \dot{\varphi}' \circ \varphi'^{-1} - \varphi'_\ast [ A(b, \varphi'; \dot{b}, \dot{\varphi}') ], 
\]
i.e. $\sigma$ is well defined.

It is not hard to show that $\sigma$ satisfies the slip boundary condition \eqref{eq:P}, so that we arrive at the following result.
\begin{prop}
	The map $\sigma$ defined in \eqref{eq:induced_interp_method} is a section of the anchor map $\rho: P \to TB$.
\end{prop}

\subsection{The induced horizontal bundle.}

The horizontal sub-bundle induced by a connection one-form, or equivalently, by a section of $\rho$ is defined as follows. At every point $(b, \varphi) \in Q$, we define a horizontal subspace $H(b, \varphi) \subset T_{(b, \varphi)} Q$ by 
\[
	H(b, \varphi) = \mathrm{ker}\, A(b, \varphi),
\]
and we let $H(Q)$ be the disjoint union of all of these spaces, so that $H(Q)$ forms a sub-bundle of $T Q$.

In our case, we have that the elements of the kernel of $A(b, \varphi)$ are the tangent vectors $(\dot{b}, \dot{\varphi})$ which satisfy by \eqref{eq:induced_interp_method} that $\sigma(b,\dot{b}) = \dot{\varphi} \circ \varphi^{-1}$.  As a consequence, there exist horizontal and vertical projectors $\mathrm{hor}, \mathrm{ver} : T Q \to TQ$ given by 
\begin{align*}
	\mathrm{hor}(b, \varphi; \dot{b}, \dot{\varphi}) & = 
		(b, \varphi; \dot{b}, \sigma(b,\dot{b}) \circ \varphi) \\
	\mathrm{ver}(b, \varphi; \dot{b}, \dot{\varphi}) & = 
		(b, \varphi; 0, \dot{\varphi} - \sigma(b,\dot{b}) \circ \varphi).
\end{align*}	
The horizontal bundle is therefore spanned by vectors of the form $(b, \varphi; \dot{b}, \sigma(b,\dot{b}) \circ \varphi)$, where $(b, \varphi) \in Q$ and $\dot{b} \in T_b \body$ are arbitrary.  

\subsection{The curvature of $\sigma$.} \label{sec:curvature}

Using Cartan's structure formula, we have that the curvature of a (right) principal connection is given by the two-form $C$ on $Q$ with values in $\mathfrak{g}$, given by 
\[
	C = \mathbf{d} A + [A, A].
\]
We recall from \eqref{MC_expression} that the connection is given by $A = \Theta_R - \sigma$, and we use the fact that $\Theta_R$ has zero curvature: $\mathbf{d} \Theta_R + [\Theta_R, \Theta_R] = 0$. From this, it easily follows that the curvature $B$ can be expressed directly in terms of $\sigma$ by 
\[
	C = -\mathbf{d} \sigma + [\sigma, \sigma].
\]
Up to a sign, this is precisely the curvature tensor that appeared in Proposition~\eqref{prop:hor_var}.

\subsection{The induced covariant derivative.}

Lastly, we use the fact that a connection one-form $A$ induces a covariant derivative on the adjoint bundle $\tilde{\mathfrak{g}}$ to justify the expression \eqref{eq_cov_deriv} given previously. It is well-known (see e.g. \cite[Lemma~2.3.4]{CeMaRa2001}) that the covariant derivative of a curve $t \mapsto (b(t), \xi_b(t)) \in \tilde{\mathfrak{g}}$ is given by 
\[
	\frac{D}{Dt}(b, \xi_b) = \left(b, \frac{d \xi_b}{dt} + 
		[A(b, \varphi)\cdot(\dot{b}, \dot{\varphi}), \xi_b] \right), 
\]
where $(b, \varphi) \in Q$ projects down onto $b \in B$, and $(\dot{b}, \dot{\varphi}) \in T_{(b, \varphi)} Q$ is such that $\dot{\varphi} \circ \varphi^{-1} = \xi_b$. Substituting the expression \eqref{conn_one_form_interpolation} for $A$ in terms of the section $\sigma$, we obtain the formula for the covariant derivative given in Proposition~\ref{prop:cov_deriv}.

\bibliographystyle{alpha}
\bibliography{JaVa2013_JGM}

\newcommand{\etalchar}[1]{$^{#1}$}
\begin{thebibliography}{KMRMH05}

\bibitem[AK92]{ArKh1992}
V~I Arnold and B~A Khesin.
\newblock {\em Topological Methods in Hydrodynamics}, volume~24 of {\em Applied
  Mathematical Sciences}.
\newblock Springer Verlag, 1992.

\bibitem[AM00]{FOM}
R~Abraham and J~E Marsden.
\newblock {\em Foundations of Mechanics}.
\newblock American Mathematical Society, 2nd edition, 2000.

\bibitem[AMR09]{MTA}
R~Abraham, J~E Marsden, and T~S Ratiu.
\newblock {\em Manifolds, Tensor Analysis, and Applications}, volume~75 of {\em
  Applied Mathematical Sciences}.
\newblock Spinger, 3rd edition, 2009.

\bibitem[Arn66]{Arnold1966}
V~I Arnold.
\newblock Sur la g\'{e}om\'{e}trie diff\'{e}rentielle des groupes de {L}ie de
  dimension infinie et ses applications \`{a} l'hydrodynamique des fluides
  parfaits.
\newblock {\em Annales de l'Institut Fourier}, 16:316--361, 1966.

\bibitem[Bat00]{Batchelor}
G~K Batchelor.
\newblock {\em An Introduction to Fluid Dynamics}.
\newblock Cambridge University Press, 2000.

\bibitem[CMR01]{CeMaRa2001}
H~Cendra, J~E Marsden, and T~S Ratiu.
\newblock {\em Lagrangian Reduction by Stages}, volume 152 of {\em Memoirs of
  the American Mathematical Society}.
\newblock American Mathematical Society, 2001.

\bibitem[FS08]{FeSt2008}
R~L Fernandes and I~Struchiner.
\newblock Lie algebroids and classification problems in geometry.
\newblock {\em S\~ao Paulo J. Math. Sci.}, 2(2):263--283, 2008.

\bibitem[GBR09]{GayBalmaz2009}
F~Gay-Balmaz and T~S Ratiu.
\newblock {The geometric structure of complex fluids}.
\newblock {\em {Advances in Applied Mathematics}}, {42}({2}):{176--275},
  {2009}.

\bibitem[GMP{\etalchar{+}}11]{Gawlik2011}
E~S Gawlik, P~Mullen, D~Pavlov, J~E Marsden, and M~Desbrun.
\newblock Geometric, variational discretization of continuum theories.
\newblock {\em Physica D: Nonlinear Phenomena}, 240(21):1724--1760, 2011.

\bibitem[HLW02]{Hairer2002}
E~Hairer, C~Lubich, and G~Wanner.
\newblock {\em Geometric Numerical Integration, Structure Preserving Algorithms
  for Ordinary Differential Equations}, volume~31 of {\em Series in
  Computational Mathematics}.
\newblock Springer Verlag, 2002.

\bibitem[JRD12]{JaRaDe2012}
H~O Jacobs, T~S Ratiu, and M~Desbrun.
\newblock On the coupling between an ideal fluid and immersed particles.
\newblock Preprint, arXiv:1208.6561v1, to appear in Physica D, 2012.

\bibitem[Kel98]{Kelly1998}
S~D Kelly.
\newblock {\em The mechanics and control of robotic locomotion with
  applications to aquatic vehicles}.
\newblock PhD thesis, California Institute of Technology, 1998.

\bibitem[KMRMH05]{Kanso2005}
E~Kanso, J~E Marsden, C~W Rowley, and J~B Melli-Huber.
\newblock {Locomotion of articulated bodies in a perfect fluid}.
\newblock {\em {Journal of Nonlinear Science}}, {15}({4}):{255--289}, {2005}.

\bibitem[KN63]{KoNo1963}
S~Kobayashi and K~Nomizu.
\newblock {\em Foundations of differential geometry. {V}ol {I}}.
\newblock Interscience Publishers, John Wiley \& Sons, 1963.

\bibitem[KSM99]{KMS99}
I~Kolar, J~Slovak, and P~W Michor.
\newblock {\em Natural Operations in Differential Geometry}.
\newblock Springer Verlag, 1999.

\bibitem[Lam45]{La1945}
H~Lamb.
\newblock {\em Hydrodynamics}.
\newblock Dover Publications, 1945.
\newblock Reprint of the 1932 Cambridge University Press edition.

\bibitem[LLM08]{LeeLeok08}
T~Lee, M~Leok, and N.~H. McClamroch.
\newblock Computational geometric optimal control of rigid bodies.
\newblock {\em Communications in Information and Systems}, 8(4):445--472, 2008.

\bibitem[LMMR86]{LeMaMoRa1986}
D~Lewis, J~E Marsden, R~Montgomery, and T~S Ratiu.
\newblock The {H}amiltonian structure for dynamic free boundary problems.
\newblock {\em Phys. D}, 18(1-3):391--404, 1986.

\bibitem[MH83]{MFOE}
J~E Marsden and T~J~R Hughes.
\newblock {\em Mathematical Foundations of Elasticity}.
\newblock Dover, 1983.

\bibitem[MS93a]{MaSc1993b}
J~E Marsden and J~Scheurle.
\newblock Lagrangian reduction and the double spherical pendulum.
\newblock {\em ZAMP}, 44:17--43, 1993.

\bibitem[MS93b]{MaSc1993a}
J~E Marsden and J~Scheurle.
\newblock The reduced {E}uler-{L}agrange equations.
\newblock {\em Fields Institute Communications}, 1:139--164, 1993.

\bibitem[MW01]{MaWe2001}
J~E Marsden and M~West.
\newblock Discrete mechanics and variational integrators.
\newblock {\em Acta Numerica}, pages 357--514, 2001.

\bibitem[Rad03]{Radford_thesis}
J~E Radford.
\newblock {\em Symmetry, reduction and swimming in a perfect fluid}.
\newblock PhD thesis, California Institute of Technology, 2003.

\bibitem[Sch95]{Schwarz1995}
G~Schwarz.
\newblock {\em Hodge {D}ecomposition---{A} {M}ethod for {S}olving {B}oundary
  {V}alue {P}roblems}, volume 1607 of {\em Lecture Notes in Mathematics}.
\newblock Springer-Verlag, Berlin, 1995.

\bibitem[SW89]{ShapereWilczek1989}
A~Shapere and F~Wilczek.
\newblock {Geometry of self-propulsion at low Reynolds number}.
\newblock {\em Journal of Fluid Mechanics}, 198:557--585, 1989.

\bibitem[Tro09]{Troyanov2009}
M~Troyanov.
\newblock On the {H}odge decomposition in {$\Bbb R^n$}.
\newblock {\em Mosc. Math. J.}, 9(4):899--926, 936, 2009.

\bibitem[VKM09]{Vankerschaver2009}
J~Vankerschaver, E~Kanso, and J~E Marsden.
\newblock The geometry and dynamics of interacting rigid bodies and point
  vortices.
\newblock {\em Journal of Geometric Mechanics}, 1(2):223--266, 2009.

\bibitem[VKM10]{VaKaMa2010}
J~Vankerschaver, E~Kanso, and J~E Marsden.
\newblock The dynamics of a rigid body in potential flow with circulation.
\newblock {\em Reg. Chaot. Dyn.}, 15(4-5):606--629, 2010.

\bibitem[Wei95]{Weinstein1995}
A~Weinstein.
\newblock Lagrangian mechanics and groupoids.
\newblock In {\em Mechanics Day}, volume~7 of {\em Fields Institute
  Proceedings}, chapter~10. American Mathematical Society, 1995.

\end{thebibliography}

\end{document}